\documentclass{amsart}
\usepackage{latexsym}
\usepackage{amsmath}
\usepackage{amssymb}
\usepackage[all,cmtip,2cell]{xy}
\UseTwocells
\usepackage{bracket}
\usepackage{fancybox} 
\usepackage{ragged2e}

\newtheorem{thm}{Theorem}[section]
\newtheorem{prop}[thm]{Proposition}
\newtheorem{cor}[thm]{Corollary}
\newtheorem{lem}[thm]{Lemma}

\theoremstyle{definition}
\newtheorem{defn}[thm]{Definition}

\theoremstyle{remark}
\newtheorem{rem}[thm]{Remark}

\newcommand{\K}{{\mathbb R}}
\newcommand{\C}{{\mathcal C}}
\newcommand{\D}{{\mathcal D}}
\newcommand{\M}{{\mathcal M}}
\newcommand{\e}{\varepsilon}
\newcommand{\R}{{\mathbb R}}
\newcommand{\mapright}[1]{%
 \smash{\mathop{%
  \hbox to 1cm{\rightarrowfill}}\limits_{#1}}}
\newcommand{\maprightd}[2]{%
 \smash{\mathop{%
  \hbox to 1.2cm{\rightarrowfill}}\limits^{#1}\limits_{#2}}}
\newcommand{\mapleft}[1]{%
 \smash{\mathop{%
  \hbox to 1cm{\leftarrowfill}}\limits_{#1}}}
\newcommand{\mapleftu}[1]{%
 \smash{\mathop{%
  \hbox to 0.8cm{\leftarrowfill}}\limits^{#1}}}
\newcommand{\maprightu}[1]{%
 \smash{\mathop{%
  \hbox to 1cm{\rightarrowfill}}\limits^{#1}}}
\newcommand{\maprightud}[2]{%
 \smash{\mathop{%
  \hbox to 1cm{\rightarrowfill}}\limits^{#1}_{#2}}}
\newcommand{\mapleftud}[2]{%
 \smash{\mathop{%
  \hbox to 1cm{\leftarrowfill}}\limits^{#1}_{#2}}}

\newcounter{eqn}[section]

\def\theeqn{\textnormal{(\thesection.\arabic{eqn})}}

\def\eqnlabel#1{%
  \refstepcounter{eqn}%
  \label{#1}%
  \leqno{\theeqn}}

\begin{document}

\title{Simplicial cochain algebras for diffeological spaces 
}

\footnote[0]{


Department of Mathematical Sciences, 
Faculty of Science,  
Shinshu University,   
Matsumoto, Nagano 390-8621, Japan   
e-mail:{\tt kuri@math.shinshu-u.ac.jp}
}

\author{Katsuhiko KURIBAYASHI}
\date{}

\begin{abstract} 
The original de Rham cohomology due to Souriau and the singular cohomology in {\it diffeology} are not isomorphic to each other in general. 
This manuscript introduces a singular de Rham complex endowed with an integration map into the singular cochain complex 
which gives the de Rham theorem for {\it every} diffeological space.  It is also proved that a morphism called the {factor map} from the original 
de Rham complex to the new one is a quasi-isomorphism for a manifold and, more general, a space with singularities. Moreover, Chen's iterated integrals are considered in a diffeological framework.  
As a consequence, we deduce that the bar complex of the original de Rham complex of a simply-connected diffeological space is quasi-isomorphic to the singular de Rham complex of the diffeological free loop space provided the factor map for the underlying diffeological space is a quasi-isomorphism. 
The process for proving the assertion yields the Leray--Serre spectral sequence and the Eilenberg--Moore spectral sequence in diffeology. 
\end{abstract}

\maketitle

\section{Introduction}
In the framework of diffeology, we prove the de Rham theorem and develop Chen's iterated integrals. 
Diffeological spaces were introduced by Souriau in the early 1980s \cite{So}. 
The notion generalizes that of a manifold. More precisely, the category $\mathsf{Mfd}$ of finite dimensional manifolds embeds into $\mathsf{Diff}$
the cateogory of diffeological spaces, which is complete, cocomplete and cartesian closed. As an advantage, we can {\it very naturally} define 
a function space consisting of smooth maps between manifolds in $\mathsf{Diff}$ so that the evaluation map is smooth without arguments on infinite dimensional 
manifolds; see \cite{IZ} and also \cite[Section 4]{C-S-W}. 
It is worth mentioning the existence of adjoint functors between $\mathsf{Diff}$ and $\mathsf{Top}$ the category of topological spaces; 
see Appendix B for a brief summary of the functors. 
Thanks to reflective properties of the adjoint functors, the full subcategory of $\Delta$-generated (numerically generated or arc-generated) topological spaces, which contains all CW-complexes,  also embeds into $\mathsf{Diff}$; see Remark \ref{rem:Delta-generated_Top}. 
Thus in the category $\mathsf{Diff}$,  
it is possible to deal simultaneously with such topological spaces and manifolds without forgetting the smooth structure. 

The category $\mathsf{Diff}$ is equivalent to the category of concrete sheaves on a concrete site; see \cite{B-H}. Moreover, Watts and Wolbert \cite{W-W} 
have shown that $\mathsf{Diff}$ is closely related to the category of differentiable stacks with adjoint functors between them. As Baez and Hoffnung have 
mentioned in \cite[Introduction]{B-H}, we can use the larger category $\mathsf{Diff}$ 
for abstract constructions and the smaller one $\mathsf{Mfd}$ for theorems that rely on 
good control over local structure. 
In this manuscript, we develop 
cohomological methods for considering the local and global nature of diffeological spaces in future work; see the end of this section 
and the diagram (6.1) in Appendix B
consisting of many categories and functors related to $\mathsf{Diff}$.

The de Rham complex for a diffeological space was introduced in \cite{So} 
and it was applied, for example to moment maps, in \cite{IZ_MM, IZ_deRham}.  
Moreover, differential forms on diffeological spaces 
are effectively used in the study of bundles in $\mathsf{Diff}$. The reader is referred to \cite{M-W, Waldorf} and \cite[Future work]{C-W_bundles} 
for such a topic.  
We observe that the de Rham complex is isomorphic to the usual one via the {\it tautological map} defined in \cite{H-V-C} 
if the input diffeological space is a manifold. 

Let $(X, \D^X)$ be a diffeological space and 
$\Omega^*(X)$ the de Rham complex due to Souriau \cite{So}. It is regarded as a counterpart 
of the de Rham complex for developing Chen's iterated integrals \cite{C} in diffeology. 
In the de Rham calculus for diffeological spaces, Iglesias-Zemmour \cite{IZ_deRham} has introduced an integration
\[
\int^{\text{IZ}} : \Omega^*(X) \longrightarrow C^*_{\text{cube}}(X), 
\] 
which is a cochain map, and investigated its properties,     
where $C^*_{\text{cube}}(X)$ denotes the normalized cubic cochain complex whose $p$-simplices are smooth maps from ${\mathbb R}^p$ to $X$. However, the de Rham theorem in diffeology, 
which asserts that such an integration map induces an isomorphism of cohomology algebras, has not yet been established.  
Indeed, the de Rham theorem described with $\int^{\text{IZ}}$ does not hold for the irrational torus; 
see \cite[Section 8]{IZ_Cech}  and Remark \ref{rem:An_example}. 

In this manuscript, the de Rham theory in $\mathsf{Diff}$ is formulated in the context of simplicial objects. In particular, we introduce 
a new de Rham complex called the {\it singular de Rham complex} of a diffeological space and a morphism of cochain algebras 
from the original de Rham complex mentioned above to the new one; see Theorem \ref{thm:main}. 
The morphism connecting the original and new de Rham complexes is called the {\it factor map} for the underlying diffeological space. 
Theorem \ref{thm:main} also shows that 
the integration map $\int^{\text{IZ}}$ is the composite of the factor map and an appropriate integration map defined on the singular de Rham complex up to homotopy. 
Moreover, the theorem allows one to deduce that 
the de Rham theorem holds for {\it every} diffeological space in our setting; 
see Corollary \ref{cor:main}. 
It is proved that 
if the given diffeological space is a smooth CW-complex in the sense of Iwase and Izumida \cite{I-I} or it stems from a parametrized stratifold in the sense of 
Kreck \cite{Kreck}, then the factor map is a quasi-isomorphism. 
We observe that the class of such stratifolds contains manifolds and spaces with singularities; see \cite[3. Examples]{Kreck}. 
Furthermore, Theorem \ref{thm:main} allows us to conclude that the integration map $\int^{\text{IZ}}$ induces a morphism of algebras on the cohomology; 
see Corollary \ref{cor:main2}. 

The Chen iterated integral map \cite{C} is deeply related to the singular de Rham complex. Let $M$ be 
a simply-connected diffeological space and $LM$ the free loop space which fits into a particular pullback diagram in $\mathsf{Diff}$. 
Then the composite of the factor map and 
a variant of the Chen iterated integral map in diffeology gives a quasi-isomorphism from the bar construction on the de Rham complex due to Souriau 
to the singular de Rham complex of $LM$ as long as the factor map for $M$ is a quasi-isomorphism; see Theorems \ref{thm:the_second_main} and 
\ref{thm:general_main}. 
We observe that these results are obtained by an adaptation of the integration map in Theorem \ref{thm:main}. 

The proof of Theorem \ref{thm:main} relies on the {\it extendability} of simplicial cochain algebras in real and rational de Rham theory in \cite{B-G, F-H-T, H, S, W}. Moreover, the method of acyclic models \cite{E-M, B-G} is applied to our setting. 
The latter half of the theorem follows from the usual argument with the Mayer--Vietoris sequence; see \cite{I-I, Haraguchi} for 
applications of the sequence in diffeology. 
We also rely on 
properties of the $D$-topology for diffeological spaces investigated 
in \cite{C-S-W} through our study. 
Thus recent results in diffeology, the classical results and well-known methods in algebraic topology serve mainly to prove our assertions. 
It may seem that there is no new idea for studying diffeology in this manuscript. 
However, the reconsiderations clarify whether or not such results and methods are applicable in a diffeological framework. If not, we have modified proofs of such classical results. 

As an advantage, we have plenty of simplicial objects for homology of diffeological spaces; see Table \ref{table1} in Section \ref{sect7}. 
Indeed, there is a suitable choice of a simplicial set and a simplicial cochain algebra with which one deduces  
the de Rham theorem of the diffeological spaces as mentioned above. It is worthwhile mentioning that 
the de Rham complex, which we choose, definitely involves the simplicial arguments and cohomology developed  
in \cite{C-W, G, Hector, Kihara1, Kihara}.  

Another choice of simplicial sets for given diffeological spaces enables us to construct 
the Leray--Serre spectral sequence (LSSS) and the Eilenberg--Moore spectral sequence (EMSS) for 
{\it fibrations} in 
$\mathsf{Diff}$; see Theorems \ref{thm:LSSS} and \ref{thm:EMSS}.
By elaborate replacement of pullbacks with homotopy pullbacks for considering smooth lifts of fibrations,  we obtain the spectral sequences and then Theorem \ref{thm:general_main}, which explains  
a diffeological version of Chen's isomorphism induced by iterated integrals. 
This is another highlight in this article; see also Remark \ref{rem:highlight}. 
Computational examples of Theorem \ref{thm:general_main} include the LSSS and the EMSS for the irrational torus and its related diffeological spaces. 
One of the computations shows that the volume forms on the odd dimensional sphere $S^{2k+1}$ with $k\geq 1$ 
determines completely the singular de Rham cohomology of the diffeological free loop space $LS^{2k+1}$ via 
Chen's iterated integral map and the factor map for the sphere; see the formula (2.4). 

We mention here that in \cite{IZ_Cech}, the \v{C}ech-de Rham spectral sequence converging 
to the \v{C}ech cohomology of a diffeological space is introduced and discussed. 
Observe that the original de Rham cohomology appears on the vertical edge of the $E_2$-term of the spectral sequence; see also \cite{IZ_2019}. 
This result is recalled in \ref{AppC} Appendix C to discuss the injectivity of the homomorphism induced by the factor map on the cohomology; 
see Proposition \ref{prop:obstruction_for_Inj}. The topic of the injectivity is dealt with in further detail in \cite{K2020_1}.  

In future work, it is expected that the local systems 
in the sense of Halperin \cite{H} which we use in the proof of 
Theorem \ref{thm:general_main} determine rational homotopy theory for non-simply connected diffeological spaces and Sullivan diffeological spaces; see \cite{G-H-T, F-H-T_II}. 
Moreover, the singular de Rham complex may produce the same argument on $1$-minimal models  as in \cite{C_extensions, C-P} in diffeology. 

We also anticipate that the singular de Rham complex yields an invariant for diffeological stacks.
In fact, as seen in \cite[Definition 2.9]{R-V}, a Lie groupoid representing a differential stack can be replaced with 
a diffeological groupoid in investigating a diffeological stack. 
The simplicial nerve of a Lie groupoid ${\mathbb X}$ defines the de Rham cohomology of the stack corresponding 
to ${\mathbb X}$; see \cite[Definition 8]{Be}. 
Thus the singular de Rham complex introduced in this manuscript would give rise to an invariant for diffeological stacks. 
This will be discussed in \cite{K2020}.


\section{The main results} \label{sect2}

We begin by recalling the definitions of a diffeological space and the de Rham complex due to Souriau. 
A good reference for the subjects is the book \cite{IZ}. We refer the reader to \cite[\S 1.2]{C} and \cite[\S 2]{B-H} for 
Chen spaces and \cite{Stacey} for the comparison between diffeological spaces and Chen spaces.  

\begin{defn}
For a set $X$,  a set  $\D^X$ of functions $U \to X$ for each open set $U$ in ${\mathbb R}^n$ and for each $n \in {\mathbb N}$ 
is a {\it diffeology} of $X$ if the following three conditions hold: 
\begin{enumerate}
\item (Covering) Every constant map $U \to X$ for all open set $U \subset {\mathbb R}^n$ is in $\D^X$;
\item (Compatibility) If $U \to X$ is in $\D^X$, then for any smooth map $V \to U$ from an open set $V  \subset {\mathbb R}^m$, the composite 
$V \to U \to X$ is also in $\D^X$; 
\item
(Locality) If $U = \cup_i U_i$ is an open cover and $U \to X$ is a map such that each restriction $U_i \to X$ is in $\D^X$, 
then the map $U \to X$ is in $\D^X$. 
\end{enumerate}
\end{defn}

We call a pair $(X, \D^X)$ consisting of a set and a diffeology a {\it diffeological space}. An element of a diffeology $\D^X$ is called a {\it plot}. 
Let $(X, \D^X)$  be a diffeological space and $A$ a subset of $X$. The {\it sub-diffeology} $\D^A$ on $A$ is defined by the initial diffeology for the inclusion 
$i : A \to X$; that is, $p\in \D^A$ if and only if $i\circ p \in \D^X$. 

\begin{defn}
Let $(X, \D^X)$ and $(Y, \D^Y)$ be diffeological spaces. A map $X \to Y$ is {\it smooth} if for any plot $p \in \D^X$, the composite 
$f\circ p$ is in $\D^Y$. 
\end{defn}

All diffeological spaces and smooth maps form a category $\mathsf{Diff}$. 
Let $(X, \D^X)$ be a diffeological space. We say that a subset $A$ of $X$ is $D$-{\it open} 
if $p^{-1}(A)$ is open for each plot $p \in \D^X$, where the domain of the plot is equipped with the standard topology. 
This topology is called the {\it $D$-topology} on $X$. 
Observe that for a subset $A$ of $X$, the $D$-topology of the sub-diffeology 
on $A$ coincides with the sub-topology of the $D$-topology on $X$ if $A$ is $D$-open; see \cite[Lemma 3.17]{C-S-W}.  

We here recall the de Rham complex $\Omega^*(X)$ of a diffeological space $(X, \D^X)$ in the sense of 
Souriau \cite{So}. 
For an open set $U$ of ${\mathbb R}^n$, let $\D^X(U)$ be the set of plots with $U$ as the domain and 
$\Lambda^*(U) = \{h : U \longrightarrow \wedge^*(\oplus_{i=1}^{n} {\mathbb R}dx_i ) \mid h \ \text{is smooth}\}$
the usual de Rham complex of $U$.  Let $\mathsf{Open}$ denote the category consisting of open sets of Euclidian spaces and smooth maps between them.  
We can regard $\D^X( \ )$ and $\Lambda^*( \ )$  as functors from $\mathsf{Open}^{\text{op}}$ to $\mathsf{Sets}$ the category of sets. 
A $p$-{\it form} is a natural transformation from $\D^X( \ )$ to $\Lambda^*( \ )$. Then the de Rham complex $\Omega^*(X)$ is the 
cochain algebra consisting of $p$-forms for $p\geq 0$; that is, $\Omega^*(X)$ is the direct sum of 
\[
\Omega^p(X) := \Set{
\xymatrix@C35pt@R10pt{
\mathsf{Open}^{\text{op}} \rtwocell^{\D^X}_{\Lambda^p}{\hspace*{0.2cm}\omega} & 
\mathsf{Sets} }
| \omega \ \text{is a natural transformation}
}
\]
with the cochain algebra structure induced by that of $\Lambda^*(U)$ pointwisely.   
We mention that the interpretation above of the de Rham complex appears in \cite{P} and \cite{I-I}. 
The de Rham complex defined above is certainly a generalization of the usual de Rham complex of a manifold. 

\begin{rem}\label{rem:tautological_map} 
Let $M$ be a manifold and $\Omega_{\text{deRham}}^*(M)$ the usual de Rham complex of $M$. 
We recall the {\it tautological map} $\theta : \Omega_{\text{deRham}}^*(M) \to \Omega^*(M)$ defined by 
\[
\theta(\omega) = \{ p^*\omega\}_{p \in {{\mathcal D}^{M}}}.
\] 
Then it follows that $\theta$ is an isomorphism of cochain algebras; see \cite[Section 2]{H-V-C}. 
\end{rem}


In order to describe the main theorem concerning the de Rham isomorphism, we recall certain associated simplicial sets and simplicial cochain algebras. 
Let ${\mathbb A}^{n}:=\{(x_0, ..., x_n) \in {\mathbb R}^{n+1} \mid \sum_{i=0}^n x_i = 1 \}$ 
be the affine space equipped with the sub-diffeology of ${\mathbb R}^{n+1}$. 
Let  $\Delta^n_{\text{sub}}$ denote 
the diffeological space, whose underlying set is the standard $n$-simplex $\Delta^n$, 
equipped with the sub-diffeology of the affine space ${\mathbb A}^{n}$.  Let $(A_{DR}^*)_\bullet$ be the simplicial cochain algebra
defined by $(A^*_{DR})_n := \Omega^*({\mathbb A}^{n})$ for each $n\geq 0$. 
We denote by $(\widetilde{A^*_{DR}})_\bullet$ the sub-simplicial cochain algebra of 
$\Omega^*(\Delta^\bullet_{\text{sub}})$ consisting of elements in the image of the map 
$j^*:  \Omega^*({\mathbb A}^n) \to \Omega^*(\Delta^n_{\text{sub}})$ 
induced by the inclusion $j : \Delta^n_{\text{sub}} \to {\mathbb A}^{n}$. 

Let $\Delta$ be the category which has posets $[n]:=\{0, 1, ..., n\}$ for $n\geq 0$ as objects and non-decreasing maps $[n] \to [m]$ for $n, m\geq 0$ as morphisms. By definition, a simplicial set is a contravariant functor from $\Delta$ to $\mathsf{Sets}$ the category of sets. 
For a diffeological space $(X, \D^X)$, let $S^D_\bullet(X)$ be the simplicial set defined by 
$ 
 S^D_\bullet(X):= \{ \{ \sigma : {\mathbb A}^n \to X \mid \sigma \ \text{is a $C^\infty$-map} \} \}_{n\geq 0}. 
$
We mention that $S^D_\bullet( \text{-})$ gives the {\it smooth singular functor} defined in \cite{C-W}.  
Moreover, let ${S^\infty_\bullet(X)}$ denote the simplicial set defined to be the image of the map
\[
{S^D_\bullet(X)} \to {S^D_\bullet(X)}_{\text{sub}}:= \{ \{ \sigma : \Delta^n_{\text{sub}} \to X \mid \sigma \ \text{is a $C^\infty$-map} \} \}_{n\geq 0}
\]
induced by the inclusion $j : \Delta^n_{\text{sub}} \to  {\mathbb A}^n$; 
see \cite{Hector} for the study of the simplicial set ${S^D_\bullet(X)}_{\text{sub}}$ in diffeology. 

Let $K$ be a simplicial set. We denote by $C^*(K)$ 
the cochain complex of maps from $K_p$ to ${\mathbb R}$ in degree $p$ and vanishing on degenerate simplices. The simplicial structure 
gives rise to the cochain algebra structure on $C^*(K)$; see \cite[10 (d)]{F-H-T} for example. 
In particular, the multiplication on $C^*(K)$ is the {\it cup product} defined by 
$f\cup g (\sigma) = (-1)^{pq}f(d_{p+1}\cdots d_{p+q}\sigma)\cdot g(d_0\cdots d_0 \sigma)$
for $f \in C^p(K)$ and $g \in C^q(K)$, where $\sigma \in K_{p+q}$ and $d_i$ denotes the $i$th face map of $K$. 
We also recall the simplicial cochain algebra $(C^*_{PL})_\bullet:=C^*(\Delta[\bullet])$, where $\Delta[n] =\text{hom}_\Delta( \text{-}, [n])$ 
is the standard simplicial set. 

For a simplicial cochain algebra $A_\bullet$,  we denote by $A(K)$ the cochain algebra  
\[
\mathsf{Sets^{\Delta^{op}}}(K, A_\bullet):= \Set{
\xymatrix@C35pt@R10pt{
\mathsf{\Delta}^{\text{op}} \rtwocell^{K}_{A_\bullet}{\hspace*{0.2cm}\omega} & 
\mathsf{Sets} }
| \omega \ \text{is a natural transformation}
}
\]
whose cochain algebra structure is induced by that of $A_\bullet$. 
Observe that, for a simplicial set $K$, the map $\nu :  C_{PL}^p(K) \to C^p(K)$ defined by $\nu(\gamma)(\sigma) = \gamma(\sigma)(id_{[p]})$ for 
$\sigma \in K_p$ gives rise to a natural isomorphism $C_{PL}^*(K) \stackrel{\cong}{\to} C^*(K)$ of cochain algebras; see \cite[Lemma 10.11]{F-H-T}.  
In particular, we have a cochain algebra of the form $A_{DR}^*(S^D_\bullet(X))$ for a diffeological space $X$. This is regarded as 
a diffeological variant of Sullivan's polynomial simplicial form for a topological space; see \cite{S}.

Moreover, we define a map 
$\alpha : \Omega^*(X) \to A_{DR}^*(S^D_\bullet(X))$ of cochain algebras by $\alpha(\omega)(\sigma) = \sigma^*(\omega)$ and define 
$\alpha' : \Omega(X) \to \widetilde{A^*_{DR}}(S^\infty_\bullet(X))$ similarly. We call the maps $\alpha$ and $\alpha'$ 
the {\it factor maps} for $X$. 
Naturality and other expected properties of the factor maps shall be discussed in Sections \ref{sect3.2} and \ref{sect6}.

One of the aims of this manuscript  is to relate cochain algebras induced by simplicilal objects mentioned above to one another.  
The following is a main theorem which describes such a relationship.

\begin{thm}\label{thm:main} \text{\em (}cf. \cite{E}, \cite[Theorem 9.7]{I-I}, \cite[Th\'eor\`emes 2.2.11, 2.2.14, 2.2.18]{G}\text{\em)} 
For a diffeological space $(X, \D^X)$, one has a homotopy commutative diagram 
\[
\xymatrix@C28pt@R25pt{
C^*({S^D_\bullet(X)}_{\text{\em sub}}) \ar[rd]_{=} \ar[r]_-{\simeq}^-\varphi & (C^*_{PL} \otimes A^*_{DR})(S^D_\bullet(X))  \ar[d]_(0.4){\text{\em mult } 
\! \circ (1\otimes \int)} & A^*_{DR}(S^D_\bullet(X)) \ar[l]^-{\simeq}_-\psi \ar[ld]^(0.35){\text{\ \ \ an ``integration"} \int}& \Omega^*(X) \ar[l]_-{\alpha} 
\ar[dl]^-{\int^{\text{\em IZ}}}
\\
&C^*({S^D_\bullet(X)}_{\text{\em sub}}) & C^*_{\text{\em cube}}(X) \ar[l]^-{\simeq}_-{l}& 
}
\]
in which $\varphi$ and $\psi$ are quasi-isomorphisms of cochain algebras and the integration map $\int$ is a morphism of cochain complexes. Here $\text{\em mult}$ denotes the multiplication on the cochain algebra $C^*({S^D_\bullet(X)}_{\text{\em sub}})$. 
Moreover, the factor map $\alpha$ is a quasi-isomorphism if $(X, \D^X)$ is a finite dimensional smooth CW-complex; 
see Definition \ref{defn:CW}, 
or stems from a parametrized stratifold via the functor $k$ in the diagram (6.1); see the three paragraphs following Definition \ref{defn:stratifold} and 
Remark \ref{rem:stratifolds}. 
\end{thm}

The homology $H(C^*({S^D_\bullet(X)}_{\text{sub}}))$ was introduced in \cite{Hector}.  
The latter half of Theorem \ref{thm:main} gives answers to \cite[Probl\`{e}me D]{Hector}, which asks conditions on a diffeological space 
for the de Rham theorem on the cochain complex due to Souriau to hold. As it turns out, the theorem holds for all diffeological spaces if 
the cochain complex is replaced with the {\it singular de Rham complex} $A^*_{DR}(S^D_\bullet( \ ))$.   

\begin{cor} \label{cor:main}
For every diffeological space $(X, \D^X)$, the integration map 
\[
\int : A^*_{DR}(S^D_\bullet(X))) \to C^*({S^D_\bullet(X)}_{\text{\em sub}})
\]
in Theorem \ref{thm:main} induces an isomorphism of algebras on the cohomology. 
\end{cor}

The right square in Theorem \ref{thm:main} is homotopy commutative.  
This leads to an obstruction for $\int^{\text{IZ}}$ to induce an isomorphism on the cohomology. 

\begin{cor}\label{cor:main2}{\em (i)} The functor $H^*({S^D_\bullet( \ )}_{\text{\em sub}})$ from the category 
$\mathsf{Diff}$ to the category of graded algebras is 
homotopy invariant and hence so is the de Rham cohomology functor $H^*(A^*_{DR}(S^D_\bullet( \ )))$. \\
 {\em (ii)} The integration map $\int^{\text{\em IZ}} : \Omega^*(X) \longrightarrow C^*_{\text{\em cube}}(X)$ induces a morphism of algebras on the cohomology. \\
{\em (iii)} The integration map $\int^{\text{\em IZ}}$ induces an isomorphism of  algebras on the cohomology if and only if the factor map $\alpha$  
in Theorem \ref{thm:main} does. 
\end{cor}

We observe that in general, the factor map does not induce an isomorphism on the cohomology. 
This is clarified in Remark \ref{rem:An_example} below. 

In \cite{I-I},  Iwase and Izumida proved the de Rham theorem for a smooth CW-complex by using cubic de Rham cohomology, which admits 
the Mayer--Vietoris sequence for {\it every} diffeological space. 
Let $X$ be a diffeological space. Then the excision axiom for the homology of $C^*({S^D_\bullet(X)}_{\text{sub}})$ holds with respect to 
the $D$-topology for $X$ and hence for the homology $H(A_{DR}^*(S^D_\bullet(X)))$; 
see Section \ref{sect5} for details. Thus we have the Mayer--Vietoris exact sequence for $H(A_{DR}^*(S^D_\bullet(X)))$ with respect to a $D$-open cover.  

An application of the integration map in Theorem \ref{thm:main} is related to Chen's iterated integrals \cite{C, C2}.  
Let $M$ be a manifold and $M^I$ the Chen space of smooth free paths. 
Then we have the pullback 
\[
\xymatrix@C35pt@R18pt{
(LM)_{\text{Chen}} \ar[r]^{\widetilde{\Delta}} \ar[d]_{ev} & M^I \ar[d]^{(\varepsilon_0, \varepsilon_1)} \\
M \ar[r]_{\Delta} & M\times M
}
\eqnlabel{add-0}
\]
of the free path fibration
 $(\varepsilon_0, \varepsilon_1) : M^I \to M\times M$ along the diagonal map $\Delta : M \to M\times M$ 
in the category $\mathsf{ChenSp}$ of Chen spaces \cite{B-H, Stacey}, 
where $\varepsilon_i$ is the evaluation map at $i$; see \cite[Section 1.2]{C}.
We also recall Chen's iterated integral 
\[
\mathsf{It} : \Omega^*(M)\otimes B(\Omega^*(M)) \to \Omega^*((LM)_{\text{Chen}})_{\text{C}}
\]
which is a morphism of chain complexes; see \cite[(2.1)]{C2} and \cite[THEOREM 4.2.1]{C}. Here 
$\Omega^* (X)_{\text{C}}$ denotes the de Rham complex for a Chen space $X$ which is a prototype 
of the original de Rham complex of a diffeological space; see \cite[2.1]{C}. Moreover, 
the source complex is the bar construction of the usual de Rham complex; see \cite[Section 2]{C_bar} and \cite[Section 1.1]{P}. 
More precisely, let $\omega_i$ be a differential $p_i$-form in $\Omega^*(M)$ for each $1\leq i \leq k$ and $q : U\to M^I$ 
a plot of the Chen space $M^I$. We define 
$\widetilde{\omega_{iq}}$ by $\widetilde{\omega_{iq}}:=(id_U \times t_i)^*q_\sharp^*\omega_i$, where $q_\sharp : U\times I \to M$ is the adjoint to 
$q$ and 
\[
t_i : {\bf \Delta}^k :=\{(x_1, ..., x_k) \in {\mathbb R}^k \mid 0\leq t_1\leq \cdots \leq t_k \leq 1\} \to I
\]
denotes the projection in the $i$th factor. 
By using integration along the fibre of the trivial fibration $U \times {\bf \Delta}^k \to U$, the iterated integral $(\int \omega_1\cdots \omega_k)_q$ is defined by
\[
(\int \omega_1\cdots \omega_k)_q := \int_{{\bf \Delta}^k} \widetilde{\omega_{1q}}\wedge \cdots \wedge \widetilde{\omega_{kq}}. 
\]
Then by definition, Chen's iterated integral $\mathsf{It}$ has the form 
\[
\mathsf{It}(\omega_0[\omega_1 | \cdots | \omega_k])= 
ev^*(\omega_0)\wedge \widetilde{\Delta^*} (\int \omega_1\cdots \omega_k). 
\eqnlabel{add-1}
\]
We can consider the same diagram as (2.1) in $\mathsf{Diff}$ in which $M$ is a general diffeological space 
and $M^I$ is the diffeological space endowed with the functional diffeology. 
The pullback is denoted by $ev : LM \to M$. Modifying the definition of Chen's iterated integral in $\mathsf{Diff}$, we have a morphism 
$\mathsf{It} : \Omega^*(M)\otimes \overline{B}(A) \to \Omega^*(LM)$ of differential graded $\Omega^*(M)$-modules; 
see Section \ref{sect7} for more details. 
We choose a cochain subalgebra $A$ of $\Omega^*(M)$ which satisfies the condition that $A^i = \Omega^i(M)$ for $i>1$, $A^0=\R$ and 
$A^1\cap d\Omega^0(M) =0$. 
The integration map in Theorem \ref{thm:main} and careful treatment of a local system in the sense of Halperin \cite{H} with respect to the evaluation map $M^I \to M\times M$ in (2.1) enable us to deduce the following pivotal theorem. 

\begin{thm}\label{thm:the_second_main} Let $M$ be a simply-connected diffeological space whose cohomology $H^i(A_{DR}(S^D_\bullet(M)))$ is of finite dimension for each $i \geq 0$. Suppose that the factor map for $M$ is a quasi-isomorphism. Then the composite 
$$\alpha \circ  \mathsf{It} : \Omega^*(M)\otimes \overline{B}(A) \to 
\Omega^*(LM) \to A^*_{DR}(S^D_\bullet(LM))$$ is a quasi-isomorphism of 
$\Omega^*(M)$-modules.  
\end{thm}

\begin{rem}\label{rem:target} Let $M$ be a Chen space; see \cite{B-H}. 
We observe that the image ${\mathcal Chen} (M)$ of Chen's iterated integral map is a cochain {\it subalgebra} of the de Rham complex denoted by 
$\Omega^*((LM)_{\text{Chen}})_{\text{C}}$. The subalgebra is isomorphic to the bar complex $\Omega^*(M)\otimes \overline{B}(A)$ mentioned above 
as a cochain complex; see \cite[Theorem 4.2.1]{C}.  Moreover, the result \cite[Theorem 0.1]{C2} 
asserts that ${\mathcal Chen} (M)$ is quasi-isomorphic to the singular cochain complex $C^*((LM)_{\text{top}})$ 
if $M$ is a manifold, where $(LM)_{\text{top}}$ denotes the function space of continuous maps from $S^1$ to $M$ with compact-open topology. 
However, it seems that the relationship on cohomology between ${\mathcal Chen} (M)$ and 
the cochain algebra $\Omega^*((LM)_{\text{Chen}})_{\text{C}}$ itself is obscure. 
On the other hand, Theorem \ref{thm:the_second_main} and its generalization 
Theorem \ref{thm:general_main} below reveal that the singular de Rham complex functor $A_{DR}^*(S^D_\bullet( \ ))$ gives rise to 
a relevant codomain of Chen's iterated integral $\mathsf{It}$. 
\end{rem}

We give computational examples of the singular de Rham cohomology algebras.

\begin{rem} \label{rem:An_example}
Let $\gamma$ be an irrational number, namely $\gamma \in {\mathbb R}\backslash {\mathbb Q}$. Consider the two dimensional torus 
$T^2 :=\{(e^{2\pi i x} , e^{2\pi i y}) \mid (x, y) \in {\mathbb R}^2 \}$ which is a Lie group, and the subgroup 
$S_\gamma := \{ (e^{2\pi i t} , e^{2\pi i \gamma t}) \mid t \in {\mathbb R}\}$ of $T^2$.   Then the {\it irrational torus} $T_\gamma$ is defined by the quotient $T^2/S_\gamma$ with the quotient diffeology. 
Since $S_\gamma$ is a dense subgroup, it follows that the topology of the homogeneous space $T^2/S_\gamma$ is trivial and hence it is contractible. 

In the category $\mathsf{Diff}$, we have a principal diffeological fibre bundle of the form $S_\gamma \to T^2 \stackrel{\pi}{\to} T_\gamma$; see \cite[8.15]{IZ}.
Let $f : M \to T_\gamma$ be a smooth map from a diffeological space $M$.  Then we obtain a principal diffeological bundle (*) : 
$S_\gamma \to M\times_{T_\gamma}T^2 \stackrel{\pi'}{\to} M$ via the pullback construction along the map $f$. 
For a diffeological space $X$, we may write $A^*(X)$ for the cochain algebra $A_{DR}^*(S^D_\bullet(X))$. A diffeological fibre bundle with a diffeological group as the fibre is a fibration in the sense of Christensen and Wu; see \cite[Propositions 4.28 and 4.30]{C-W}. Then the Leray--Serre spectral sequence in Theorem \ref{thm:LSSS} for the fibration (*) allows us to deduce that $\pi'$ gives rise to an isomorphism 
\[
\xymatrix@C18pt@R0pt{
 (\pi')^* : H^*(A^*(M)) \ar[r]^{\cong} &  H^*(A^*(M\times_{T_\gamma}T^2))
}
\eqnlabel{add-2}
\] 
of algebras.  We observe that the local system ${\mathcal H}^*(S_\gamma)$ is simple. In fact, 
the fibre $S_\gamma$ is diffeomorphic to $({\mathbb R}, +)$ 
as a Lie group and hence it is contractible in $\mathsf{Diff}$. Thus we have isomorphisms 
$
\xymatrix@C18pt@R0pt{
\wedge (dt_1, dt_2) \cong H^*_{\text{deRham}}(T^2) \ar[r]^-{\alpha}_-{\cong}& H^*(A^*(T^2))  & H^*(A^*(T_\gamma)). \ar[l]_-{(\pi)^*}^-{\cong}
}
$
Here $H^*_{\text{deRham}}(T^2)$ is the usual de Rham cohomology algebra of the manifold $T^2$ and $dt_i$ denotes the image $(pr_i)^*(d\theta)$ of the volume form $d\theta \in H^1(T)$ of the one dimensional torus $T$ 
by the map $(pr_i)^*$ induced by the projection $pr_i$ 
on $i$th factor.  
We recall the diffeomorphism $\psi : {\mathbb R}/({\mathbb Z}+\gamma {\mathbb Z}) \to T_\gamma$ defined by 
$\psi(t) = (0, e^{2\pi i t})$ in \cite[Exercise 31, 3)]{IZ}. Then the isomorphism 
$p^* : \Omega^*({\mathbb R}/({\mathbb Z}+\gamma {\mathbb Z})) \stackrel{\cong}{\to} (\wedge^*({\mathbb R}), d \equiv 0)$
induced by the subduction $p : {\mathbb R} \to {\mathbb R}/({\mathbb Z}+\gamma {\mathbb Z})$ in 
\cite[Exercise 119]{IZ} 
fits into the commutative diagram 
\[
\xymatrix@C35pt@R15pt{
\Omega^*(T) \ar[r]^{\rho^*} & \Omega^*({\mathbb R}^1) \\
\Omega^*(T^2) \ar[u]^{(in_2)^*} & \wedge^*({\mathbb R}^1) \ar@{>->}[u]\\
\Omega^*(T_\gamma) \ar[r]^-{\psi^*}_-{\cong} \ar[u]^{\pi^*} & \Omega^*({\mathbb R}/({\mathbb Z}+\gamma {\mathbb Z})), \ar[u]_{p^*}^{\cong}
}
\]
where $\rho : {\mathbb R} \to T$ denotes the projection and $in_2$ is the inclusion 
in the second factor. 
We may assume that $\rho^*(d\theta)=dt$ for the constant differential form $dt \in \wedge^1({\mathbb R})$. 
Then it follows that 
$\pi^*(\psi^*)^{-1}(p^*)^{-1}dt = dt_2$.   
In fact, we write $a_1 dt_1+a_2 dt_2$ for $\pi^*(\psi^*)^{-1}(p^*)^{-1}dt$. The commutativity of the diagram above implies that $a_2=1$. 
We see that $(a_1 dt_1+a_2 dt_2)^2 = (\pi^*(\psi^*)^{-1}(p^*)^{-1}dt)^2=0$ and hence $a_1 = 0$. 
Moreover, the naturality of the factor map $\alpha$ in Theorem \ref{thm:main} gives a commutative diagram 
\[
\xymatrix@C35pt@R18pt{
A^*(T^2) & \Omega^*(T^2) \ar[l]_{\alpha}^{\simeq} \\
A^*(T_\gamma) \ar[u]^{\pi^*}_{\simeq}& \Omega^*(T_\gamma).  \ar[l]^{\alpha} \ar[u]_{\pi^*} 
}
\]
Thus we see that the morphism $H(\alpha)$ of algebras from the original de Rham cohomology of $T_\gamma$ to the singular de Rham cohomology is a non-surjective monomorphism; see \ref{AppC} Appendix C for a more general consideration on the injectivity. 
Here it is worth to mention that the irrational torus $T_\gamma$ is not homotopy equivalent to the torus $T^2$. In fact, it follows that 
the cohomology $H^*(\Omega^*( \ ))$ is homotopy invariant for diffeological spaces; see \cite[6.88]{IZ}. While $H^1(\Omega^*(T^2)) \cong {\mathbb R}^2$, 
the computation above shows that $p^* : H^1(\Omega^*(T_\gamma)) \stackrel{\cong}{\to} {\mathbb R}^1$.

We recall the isomorphism $(\pi')^*$ in (2.3). Suppose that $M$ is simply connected. Then the comparison of the EMSS's  in Theorem \ref{thm:EMSS} for 
$LM$ and 
$L(M\times_{T_\gamma}T^2)$ allows us to obtain an algebra isomorphism  
\[
(L\pi')^* : H^*(A^*(LM)) \stackrel{\cong}{\longrightarrow} 
H^*(A^*(L(M\times_{T_\gamma}T^2)).
\]
Thus if $H^*(A^*(M))\cong H^*(A^*(S^{2k+1}))$ as an algebra with $k\geq 1$, then we see that 
\[
\xymatrix@C18pt@R0pt{
H^*(A^*(L_{\text{free}}(M\times_{T_\gamma}T^2))) \cong \wedge(\alpha \circ  \mathsf{It}({(\pi')}^*(\omega)))
\otimes {\mathbb R}[\alpha \circ  \mathsf{It} (1\otimes {(\pi')}^*(\omega))]
}
\eqnlabel{add-31}
\] 
as an $H^*(A^*(M))$-algebra, where $\omega$ is the volume form of $M$. 
In fact, the result follows from Theorem \ref{thm:the_second_main} and \cite[Theorem 2.1 and Corollary 2.2]{K1996}. 
This is the first computational example of the singular de Rham cohomology algebra of a diffeological loop space. Moreover, 
Corollary \ref{cor:main} asserts that we can also determine the singular cohomology algebra of $L(M\times_{T_\gamma}T^2)$ 
with coefficients in ${\mathbb R}$. 
%
\end{rem}


The rest of this manuscript  is organized as follows. 
In order to prove Theorem \ref{thm:main}, the extendability of the simplicial cochain algebras $A^*_{DR}$ and  $\widetilde{A^*_{DR}}$ 
is verified in Section \ref{sect3}. 
Section \ref{sect3.2} explains the integration map and the factor map $\alpha$ in Theorem \ref{thm:main}. 
The theorem and Corollaries \ref{cor:main} and \ref{cor:main2} 
are proved in Section \ref{sect5}.  In Section \ref{sect7}, after modifying Chen's iterated integrals from a diffeological point of view,  
we prove Theorem \ref{thm:the_second_main} as a corollary of a more general result (Theorem \ref{thm:general_main}). 

In Appendix A, we recall the acyclic model theorem for cochain complexes.
In Appendix B, the definitions of a smooth CW complex and a parametrized stratifold are recalled. 
Moreover, we briefly summarize functors between categories related to our subjects 
in this article. In Appendix C, we discuss the injectivity of $H(\alpha)$ induced by the factor map $\alpha$ on cohomology 
with the \v{C}ech--de Rham  spectral sequence introduced 
in \cite{IZ_Cech, IZ_2019}.

\section{Preliminaries}
\subsection{Extendability of the simplicial cochain algebra  $\widetilde{A^*_{DR}}$.}\label{sect3}
We begin with the definition of the extendability of a simplicial object. 
The notion plays an important role in the proofs of Theorems \ref{thm:main} and \ref{thm:the_second_main}.  
 
\begin{defn} \label{defn:extendableOb}
A simplicial object $A$ in a category $\C$ is {\it extendable} if for any $n$, every subset set ${\mathcal I} \subset \{0, 1, ..., n\}$ and any elements 
$\Phi_i \in A_{n-1}$ for $i  \in {\mathcal I}$ which satisfy the condition that $\partial_i\Phi_j = \partial_{j-1} \Phi_i$ for $i <j$, there exists an element 
$\Phi \in A_n$ such that $\Phi_i = \partial_i \Phi$ for $i \in {\mathcal I}$.  
\end{defn}

With Remark \ref{rem:tautological_map} in mind, we prove the following lemma due to Emoto \cite{Emoto}. Though 
the proof indeed uses the same strategy as in \cite[13.8 Proposition]{H} and \cite[Lemma 10.7 (iii)]{F-H-T}, 
we introduce it for the reader. 

\begin{lem} \label{lem:extendability} 
The simplicial differential graded algebra $(\widetilde{A^*_{DR}})_\bullet$ is extendable. 
\end{lem}

\begin{proof}
Let ${\mathcal I}$ be a subset of $\{0, 1, ..., n\}$ and $\Phi_i $ an element in $(\widetilde{A^*_{DR}})_{n-1}$ for $i \in {\mathcal I}$. We assume that 
$\partial_i\Phi_j = \partial_{j-1} \Phi_i$ for $i <j$.  We define inductively elements $\Psi_r \in (\widetilde{A^*_{DR}})_{n}$ for $-1 \leq r \leq n$ which satisfy the condition that (*): $\partial_i \Psi_r= \Phi_i$ if $i \in {\mathcal I}$ and $i\leq r$. Put $\Psi_{-1} =0$ and suppose that $\Psi_{r-1}$ is given with (*). 
Define a smooth map $\varphi : {\mathbb A}^n-\{v_r\} \to {\mathbb A}^{n-1}$ by 
\[
\varphi(t_0, t_1, ..., t_n) = \big(\frac{t_0}{1-t_{r}}, .., \frac{t_{r-1}}{1-t_r}, \frac{t_{r+1}}{1-t_r}, ..., \frac{t_n}{1-t_r}\big), 
\]
where $v_r$ denotes the $r$th vertex. The map $\varphi$ induces a morphism 
$\varphi^* :  \Omega^*({\mathbb A}^{n-1}) \to \Omega^*({\mathbb A}^n - \{v_r\})$  
of cochain algebras. For an element $u$ in  $(\widetilde{A^*_{DR}})_{n-1}$, we write $u'$ for an element in $\Omega^*({\mathbb A}^{n-1})$ with 
$j^*(u')=u$. If $r$ is not in ${\mathcal I}$, we define $\Psi_r$ by $\Psi_{r-1}$. In the case where $r \in {\mathcal I}$, we consider the element 
$\Phi_r' -\partial_r\Psi_{r-1}'$ in $\Omega^*({\mathbb A}^{n-1})$. 

Let $k_r : {\mathbb A}^n \to {\mathbb A}$ be the projection in the $r$th factor and $\rho$ a cut-off function with $\rho(0) =1$ and $\rho(1) = 0$. 
We observe that $(\rho\circ k_r)$ is in $\Omega^0({\mathbb A}^n)$.  
Then the action of $(\rho\circ k_r)$ on $\Omega^*({\mathbb A}^n - \{v_r\})$ defined by the pointwise multiplication 
gives rise to a linear map 
$
(\rho\circ k_r)\star \text{-} : \Omega^*({\mathbb A}^n - \{v_r\}) \to \Omega^*({\mathbb A}^n). 
$
We see that the map $(\rho\circ k_r)\star \text{-}$ fits in the commutative diagram  
\[
\xymatrix@C35pt@R20pt{
\Omega^*({\mathbb A}^{n-1}) \ar[r]^-{\varphi^*} \ar[d]_{\partial_i}& \Omega^*({\mathbb A}^n - \{v_r\}) \ar[r]^-{(\rho\circ k_r)\star } 
\ar[d]_{\partial_i} & \Omega^*({\mathbb A}^n) \ar[d]^{\partial_i} \\
   \Omega^*({\mathbb A}^{n-2}) \ar[r]_-{\varphi^*} & \Omega^*({\mathbb A}^{n-1} - \{v_{r-1}\}) \ar[r]_-{(\rho\circ k_{r-1})\star } 
  & \Omega^*({\mathbb A}^{n-1}) 
}
\]
for $i < r$. 
Define $\Psi \in (\widetilde{A^*_{DR}})_n$ by 
$
\Psi := j^*\big((\rho\circ k_r)\star \varphi^*(\Phi_r' -\partial_r\Psi_{r-1}') \big). 
$ 
Since $\partial_i(\Phi_r-\partial_r\Psi_{r-1}) = \partial_{r-1}(\Phi_i -\partial_i\Psi_{r-1})=0$ by assumption for $i <r$, 
it follows from the commutative diagram above that for $i <r$, 
\begin{align*}
\partial_i\Psi &= \partial_ij^*((\rho\circ k_r)\star \varphi^*(\Phi_r' -\partial_r\Psi_{r-1}'))  \\
                    &= j^*((\rho\circ k_{r-1})\star \partial_i\varphi^*(\Phi_r' -\partial_r\Psi_{r-1}'))) 
                    = (\rho\circ k_{r-1})\star j^*(\partial_i\varphi^*(\Phi_r' -\partial_r\Psi_{r-1}')))  \\
                    &= (\rho\circ k_{r-1})\star j^*\varphi^*\partial_i(\Phi_r' -\partial_r\Psi_{r-1}')) 
                    = (\rho\circ k_{r-1})\star \varphi^*\partial_i(\Phi_r -\partial_r\Psi_{r-1}))=0.  
\end{align*}
The third and fifth equalities follow from the commutativity of the diagram 
\[
\xymatrix@C35pt@R16pt{
\Omega^*({\mathbb A}^{n-2}) \ar[r]^-{\varphi^*} \ar[d]_{j^*}& \Omega^*({\mathbb A}^{n-1} - \{v_{r-1}\}) \ar[r]^-{(\rho\circ k_{r-1})\star } 
\ar[d]_{j^*} & \Omega^*({\mathbb A}^{n-1}) \ar[d]^{j^*} \\
   \text{Im} \ j^* \ar[r]_-{\varphi^*} \ar@<-0.3ex>@{^{(}->}[d] &   \text{Im} \ j^* \ar[r]_-{(\rho\circ k_{r-1})\star } \ar@<-0.3ex>@{^{(}->}[d]
  & (\widetilde{A^*_{DR}})_{n-1} \\
  \Omega^*(\Delta_\text{sub}^{n-2}) \ar[r]_-{\varphi^*}& \Omega^*(\Delta_\text{sub}^{n-1}- \{v_{r-1}\}) . 
}
\]

Since $\partial_r (\rho \circ k_r) =1$ and $\varphi \circ \partial^r =id_{{\mathbb A}^n}$, it follows that the diagram 
\[
\xymatrix@C35pt@R18pt{
\Omega^*({\mathbb A}^{n-1}) \ar[r]^-{\varphi^*} \ar[rd]_{id}& \Omega^*({\mathbb A}^{n} - \{v_{r}\}) \ar[r]^-{(\rho\circ k_{r})\star } 
\ar[d]_{\partial_{r}} & \Omega^*({\mathbb A}^{n}) \ar[d]^{\partial_r} \\
  & \Omega^*({\mathbb A}^{n-1}) \ar[r]_-{id} 
  & \Omega^*({\mathbb A}^{n-1})
}
\]
is commutative. Thus we have $\partial_r\Psi = \Phi_r -\partial_r\Psi_{r-1}$. 
It turns out that $\partial_j(\Psi + \Psi_{r-1})=\Phi_j$ for $j \in {\mathcal I}$ and $j\leq r$. 
This completes the proof. 
\end{proof}

We verify that the Poincar\'e lemma holds for $(\widetilde{A_{DR}})_n$. 

\begin{lem}\label{lem:PoincareLemma}
One has $H^*((\widetilde{A_{DR}})_n)= {\mathbb R}$ for any $n\geq 0$.  
\end{lem}

\begin{proof} Let $X$ and $Y$ be diffeological spaces. Let $X^{\mathbb R}:=\text{map}({\mathbb R}, X)$ denote the function space with functional diffeology. 
We first remark that the chain homotopy operator $K_X : \Omega (X) \to \Omega^{*-1}(X^{\mathbb R})$ defined in \cite[6.83]{IZ} is natural with respect to smooth maps. Moreover a smooth homotopy  $H : Y\times {\mathbb R} \to X$ from $H_0$ to $H_1$ gives a cochain homotopy defined by $\widetilde{K}_X:=ad(H)^*\circ K_X: \Omega(X) \to\Omega^{*-1}(Y)$ from $H_1^*$ to $H_1^*$, where $ad(H) : Y \to X^{\mathbb R}$ is the adjoint map to $H$. 

We choose a smooth contraction map $\Delta_\text{sub}^n \to \Delta_\text{sub}^n$ that extends to one on the affine space $\mathbb{A}^n$. 
Moreover, we have a smooth homotopy between the contraction and the identity map on $\mathbb{A}^n$ whose restriction is such a homotopy 
on $\Delta_\text{sub}^n$. 
The homotopies give rise to cochain homotopies $\widetilde{K}_{{\mathbb A}^n}$ and $\widetilde{K}_{\Delta_\text{sub}^n}$ which 
fit into a commutative diagram 
\[
\xymatrix@C25pt@R18pt{
\Omega^*(\mathbb{A}^n) \ar[r]^{j^*} \ar[d]_{\widetilde{K}_{{\mathbb A}^n}} & \Omega^*(\Delta^n_{\text{sub}}) \ar[d]^{\widetilde{K}_{\Delta_\text{sub}^n}} \\
\Omega^{*-1}(\mathbb{A}^n) \ar[r]^{j^*} & \Omega^{*-1}(\Delta^n_{\text{sub}}). 
}
\]
Thus the cochain homotopy $\widetilde{K}_{\Delta_\text{sub}^n}$ restricts to $(\widetilde{A_{DR}})_n$. As a consequence, we have the result. 
\end{proof}

Thanks to the extendability and the Poincar\'e lemma for the simplicial cochain algebra $(\widetilde{A_{DR}})_\bullet$, 
the same argument as in the proof of \cite[Theorem 10.9]{F-H-T}, which gives quasi-isomorphisms between $C^*(X ; {\mathbb Q})$ and the rational de Rham complex $A_{PL}(X)$ for a space $X$, works well in our setting. 
In fact, replacing the simplicial complex $(A_{PL})_\bullet$ of polynomial differential forms in the proof with $(\widetilde{A_{DR}})_\bullet$, 
we have

\begin{prop} \label{prop:3.4}
Let  $K$ be a simplicial set. Then there is a sequence of quasi-isomorphisms
\[
\xymatrix@C25pt@R25pt{
C^*(K) & C_{PL}^*(K) \ar[l]^{\nu}_{\cong} \ar[r]^-{\simeq}_-\varphi & (C_{PL} \otimes \widetilde{A_{DR}})^*(K) & 
\widetilde{A_{DR}^*}(K), \ar[l]_-{\simeq}^-\psi 
}
\]
where $\varphi$ and $\psi$ are defined by $\varphi(\gamma)=\gamma \otimes 1$ and $\psi(\omega)=1\otimes \omega$, respectively. 
\end{prop}

\begin{proof} 
We recalled the definition of the isomorphism $\nu$ before Theorem \ref{thm:main}. Since $(C_{PL})_\bullet$ and $(\widetilde{A_{DR}})_\bullet$ are extendable, it follows 
from the proof of \cite[Lemma 10.12 (iii)]{F-H-T} that so is $(C_{PL}\otimes \widetilde{A_{DR}})_\bullet$. Then the result \cite[Proposition 10.5]{F-H-T} yields that 
$\varphi$ and $\psi$ are quasi-isomorphisms. 
\end{proof}

Let $(\Omega_{\text{deRham}}^*)_\bullet$ be the simplicial cochain algebra 
defined by $(\Omega_{\text{deRham}}^*)({\mathbb A}^n)$ in degree $n$. The affine space ${\mathbb A}^n$ is a manifold diffeomorphic to ${\mathbb R}^n$ 
with the projection $\pi : {\mathbb A}^n \to {\mathbb R}^n$ defined by $\pi (x_0, x_1, ..., x_n) =(x_1, ..., x_n)$. Observe that the sub-diffeology of ${\mathbb A}^n$ 
described in Section \ref{sect2} coincides with the diffeology that comes from the structure of the manifold ${\mathbb A}^n$ mentioned above. 
We see that the extendability is satisfied and that the Poincar\'e lemma holds for $(A^*_{DR})_\bullet$. 
In fact, these results follow from the proofs of Lemmas \ref{lem:extendability} and \ref{lem:PoincareLemma}. 
Therefore, Proposition \ref{prop:3.4} is also valid after replacing $(\widetilde{A_{DR}^*})_\bullet$ with the simplicial cochain algebra $(A_{DR}^*)_\bullet$. 
Thus the result \cite[Proposition 10.5]{F-H-T} enables us to deduce the following corollary. 

\begin{cor}\label{cor:RHT}
For a simplicial set $K$, there is a sequence of quasi-isomorphisms 
\[
\xymatrix@C25pt@R20pt{
A_{PL}^*(K)\otimes_{\mathbb Q} {\mathbb R}  \ar[r]_-{\simeq}^-s & \Omega_{\text{\em deRham}}^*(K) \ar[r]^-{\theta}_-{\cong} &A_{DR}^*(K) 
\ar[r]^-{(j^*)_*}_-{\simeq} &  (\widetilde{A^*_{DR}})(K),
}
\]
where $s$ denotes the map induced by the inclusion $(A^*_{PL})_\bullet \to (\Omega_\text{\em deRham}^*)_\bullet$ and $\theta$ is the tautological map in Remark \ref{rem:tautological_map}.
\end{cor}

\subsection{The factor map and an integration map}\label{sect3.2}
In this subsection, for a map $\tau : [n] \to [m]$ in $\Delta$, we use the same notation $\tau$ for the affine maps ${\mathbb A}^n \to {\mathbb A}^m$ and 
$\Delta^n \to \Delta^m$ induced by the non-decreasing map. 
We recall the factor map 
$\alpha' : \Omega^*(X) \to  \mathsf{Sets^{\Delta^{op}}}(S^\infty_\bullet(X),  \widetilde{A_{DR}})=:\widetilde{A_{DR}}(S^\infty_\bullet(X))$
defined by 
$\alpha'(\omega)(\sigma) = \sigma^*(\omega)$ for $\sigma \in S^\infty_l(X)$. 
Let $j : \Delta_\text{sub}^l \to \mathbb{A}^l$ be the inclusion. 
By definition, we see that $\sigma = \widetilde{\sigma}\circ j$ for some smooth map $\widetilde{\sigma} : \mathbb{A}^l \to X$. 
Then it is readily seen that $\sigma^*(\omega)=j^*(\widetilde{\sigma}^*\omega)$ and hence $\alpha'$ is well defined. 
A standard calculation allows us to conclude that $\alpha'$ is a morphism of cochain alegbras. 
Moreover, it follows that $\alpha'$ is natural with respect to diffeological spaces. Indeed, for a morphism $Y \to X$ in $\mathsf{Diff}$, we have 
$((f_*)^*\alpha' (\omega))(\sigma_Y)= \alpha'(\omega)(f_*\sigma_Y) =\alpha'(\omega)(f\circ \sigma_Y) = (f\circ \sigma_Y)^*\omega =\sigma_Y^*(f^*\omega) =
\alpha'(f^*\omega)(\sigma_Y)$, where $\omega \in \Omega^*(X)$ and $\sigma_Y \in S^\infty_\bullet(Y)$. Thus,  the map $\alpha'$ gives rise to a natural transformation 
$\alpha' : \Omega^*(\text{-}) \to \widetilde{A_{DR}}(S^\infty_\bullet(\text{-}))$.  
We also define a natural transformation $\alpha : \Omega^*(\text{-}) \to  A_{DR}^*(S^D_\bullet(\text{-}))$ in the same way as that for $\alpha'$. Observe that the natural transformation gives the {\it factor map} $\alpha$ described in the Section 2.  

We define an integration map $\int_{\Delta^p} : (\widetilde{A_{DR}^p)}_p \to \mathbb{R}$ by $\int_{\Delta^p}\omega = \int_{\Delta^p}\eta$ choosing 
$\eta \in \Omega^p_\text{deRham}(\mathbb{A}^p)$ with $j^*\theta(\eta) = \omega$. The definition of the integration is independent on the choice of the element $\eta$. In fact, for $\eta$ and $\eta'$ with $j^*\theta (\eta) = \omega = j^*\theta(\eta')$, we see that 
$j^*\theta (\eta) (\tau)= j^*\theta(\eta') (\tau)$ for the inclusion $\tau : (\Delta^p)^\circ \to \Delta_\text{sub}^p$ from the interior of $\Delta^p$ which is a plot in $\mathcal{D}^{{\mathbb A}^p}$. 
This implies that $\eta\circ(j\circ \tau) = \eta'\circ(j\circ \tau)$. Since $\eta$ and $\eta'$ are smooth maps on $\mathbb{A}^p$, it follows that 
$\eta = \eta'$ on $\Delta^p$. Then  $\int_{\Delta^p}\eta =  \int_{\Delta^p}\eta'$. 

We define a map $\int : (\widetilde{A_{DR}^*)}_\bullet \to (C_{PL}^*)_\bullet = C^*(\Delta[\bullet])$ by 
\[
(\int \gamma)(\sigma)= \int_{\Delta^p}\sigma^*\gamma
\eqnlabel{add-3}
\] 
for $\gamma \in  (\widetilde{A_{DR}^p)}_n$, 
where $\sigma : \Delta^p \to \Delta^n$ is the affine map induced by $\sigma : [p] \to [n]$. Since the affine map $\sigma$ is extended to an affine map 
from $\mathbb{A}^p$ to  $\mathbb{A}^n$, it follows that  $\sigma^*\gamma$ is in $(\widetilde{A_{DR}^p)}_p$. 
Then, the map $\int$ is a cochain map. To see this, let $\sigma$ be an element in $\Delta[n]_p$ and 
$\gamma'$ a form in $(\widetilde{A^{p-1}_{DR}})_n$ with $\sigma^*(\gamma')=j^*\theta(\eta')$ for some $\eta'\in \Omega_{DR}^{p-1}(\mathbb{A}^p)$. 
We have 
\begin{align*}
(\int d\gamma')(\sigma) &= \int_{\Delta^p}\sigma^*(d\gamma') =  \int_{\Delta^p}d(\sigma^*\gamma') = \int_{\Delta^p}d(\eta') \\
              &= \int_{\partial\Delta^p}\eta' = \sum_i(-1)^i\int_{\Delta^{p-1}} d_i^*\eta' =  \sum_i(-1)^i\int_{\Delta^{p-1}} d_i^* \sigma^*\gamma' \\
              &=  \sum_i(-1)^i\int_{\Delta^{p-1}} (\sigma \circ d_i)^*\gamma' = (d(\int \gamma'))(\sigma). 
\end{align*}
The fourth and fifth equalities follows from Stokes' theorem for a manifold; see \cite[V. Sections 4 and 5]{B} for example. We show that the integration 
is a morphism of simplicial sets. Let $\sigma : [p] \to [m]$ and $\tau : [m] \to [n]$  be a map in $\Delta$.  For a $\gamma \in 
(\widetilde{A^{p}_{DR}})_n$, we take a differential form $\eta$ in $\Omega_{DR}^{p}(\mathbb{A}^p)$ with 
$(\tau\circ \sigma)^*\gamma = j^*\theta (\eta)$. Then it follows that 
\begin{align*}
\tau^*(\int \gamma) (\sigma) &= (\int \gamma) (\tau\circ \sigma) = \int_{\Delta^p}\eta = \int_{\Delta^p}\sigma^*(\tau^*\gamma) 
=(\int \tau^*\gamma)(\sigma). 
\end{align*}
The third equality follows from $\sigma^*(\tau^*\gamma) = j^*\theta (\eta)$. As a consequence, we see that $\int$ is a morphism of 
simplicial differential graded {\it modules}.

Let $1$ be the unit of  $\widetilde{A_{DR}^*}_\bullet$, which is in  
$\mathsf{Diff}({\mathbb A}^n,  {\mathbb R}) =\Omega_{DR}^0({\mathbb A}^n)=(\widetilde{A_{DR}^0})_n$. 
Then we see that $\int 1 =1$ in $(C_{PL}^0)_n$ for $n\geq 0$. This yields the commutative diagram 
\[
\xymatrix@C45pt@R20pt{
(C_{PL}^*)_\bullet \ar[rd]_{=} \ar[r]^-\varphi & (C_{PL} \otimes \widetilde{A_{DR}})_\bullet^* 
\ar[d]_(0.4){\text{mult } \! \circ (1\otimes \int)} & \widetilde{A_{DR}^*}_\bullet \ar[l]_-\psi 
\ar[ld]^{\int}\\
&(C_{PL}^*)_\bullet.  &  
}
\eqnlabel{add-4}
\]

The argument above in this subsection works well for the simplicial cochain algebra ${A_{DR}}^*_\bullet$. 
In consequence, in the diagram above, the commutativity remains valid even if $\widetilde{A_{DR}^*}_\bullet$ is replaced with the simplicial cochain algebra ${A_{DR}}^*_\bullet$; 
see \cite[Remark, page 130]{F-H-T} for the same triangles as above for the polynomial de Rham complex $A_{PL}^*$.  

\section{Proofs}\label{sect5}
\subsection{Proofs of Theorem \ref{thm:main} and corollaries \ref{cor:main} and \ref{cor:main2}}\label{sub4.1}
We may write $H_*^D(X)$ for the homology of ${\mathbb Z}{S^D_\bullet(X)}_{\text{sub}}$ which is the chain complex with coefficients in ${\mathbb Z}$ 
induced by the simplicial set ${S^D_\bullet(X)}_{\text{sub}}$. 

We next discuss the homotopy axiom for the homology $H_*^D(X)$.   
Let $f$ and $g$ be  smooth maps from $X$ to $Y$ which are homotopic smoothly in the sense of Iglesias-Zemmour \cite{IZ}. 
Then the homomorphisms $f_*$ and $g_*$ induced on the homology coincide with each other: $f_*=g_* :  H_*(X) \to H_*(Y)$. 
The construction of the chain homotopy is almost verbatim a repetition of the usual one for singular chain. Observe that the proof uses the fact that 
$\Delta_\text{sub}^n \times {\mathbb R} \cong (\Delta^n \times {\mathbb R})_\text{sub}$ as a diffeological space and the following lemma; 
see, for example, \cite[1.10 Theorem]{V}.  

\begin{lem} \label{lem:acyclic} If $X$ is a convex subset of ${\mathbb R^k}$ with sub-diffeology, then 
the $n$th homology $H_n(S^D_\bullet(X))$ is trivial for $n >0$ .The same assertion is valid for the functors 
${\mathbb Z}{S^D_\bullet( \text{-})}_{\text{\em sub}}$ and ${\mathbb Z}{S^\infty_\bullet(\text{-})}$. 
\end{lem}

\begin{proof} For a smooth simplex $\sigma$ in  ${S^D_n(X)}$ and a point $v \in X$, we define a {\it cone} 
$K_v(\sigma)$ by 
\[
K_v(\sigma)(t_0 ,...., t_{n+1}) =  
\begin{cases}
    \rho(1- t_{0})\sigma(\frac{t_1}{1-t_{0}}, ..., \frac{t_n+1}{1-t_{0}}) + \tau(1-t_{0})v  & \text{for} \ t_{0} \neq 1 \\
    v  & \text{for} \ t_{0} =1, 
  \end{cases}
\]
where $\rho$ is a cut-off function with $\rho(0)= 0$, $\rho(1)=1$ and $\tau$ is the smooth function defined by $\tau= 1- \rho$. We see that 
$K_v(\sigma)$ is in $S^D_{n+1}(X)$. By extending $K_v$  linearly, we have a homomorphism 
$K_v :  {\mathbb Z}{S^D_n(X)} \to {\mathbb Z}{S^D_{n+1}(X)}$. 
This gives a homotopy between the identity and the zero map; 
see the proof of \cite[1.8 Lemma]{V}. The same argument as above works in ${S^D_n(X)}_{\text{sub}}$. 

As for the functor ${S^\infty_\bullet(\text{-})}$, we can define a cone 
$K_v  : {\mathbb Z}{S^\infty_n(X)} \to {\mathbb Z}{S^\infty_{n+1}(X)}$ with  an extension 
$\widetilde{\sigma} : {\mathbb A}^n \to X$ for $\sigma : \Delta_{\text{sub}}^n \to X$. This gives a homotopy between the identity and the zero map. 
\end{proof}

To apply the method of acyclic models in proving Theorem \ref{thm:main}, we need the following result. 

\begin{lem} \label{lem:representable}
Let ${\mathcal M}$ be the set of convex subsets of ${\mathbb R}^k$ for $k \geq 0$. Then the three functors 
${\mathbb Z}{S^D_n( \text{-})}$, ${\mathbb Z}{S^D_n( \text{-})}_{\text{\em sub}}$ and ${\mathbb Z}{S^\infty_n(\text{-})}$ are representable for ${\mathcal M}$ 
in the sense of Eilenberg--Mac Lane for each $n$; see \cite[Definition, page 189]{E-M}. 
\end{lem}

\begin{proof} Let $\widetilde{{\mathbb Z}{S^\infty_n}}(X)$ be the free abelian group generated by 
$\amalg_{M \in {\mathcal M}}({\mathbb Z}S^\infty_n(M) \times 
\text{Hom}_{\mathsf{Diff}}(M, X))$.  We define a map 
$
\Phi : \widetilde{{\mathbb Z}{S^\infty_n}}(X) \to {\mathbb Z}S^\infty_n(X) 
$
by $\Phi (m, \phi) = \phi_*m = m\circ \phi$. For $m \in {\mathbb Z}S^\infty_n(X)$, one has an extension 
$\widetilde{m} : {\mathbb A}^n \to X$ by definition. Define a map $\Psi :  {\mathbb Z}S^\infty_n(X)  \to \widetilde{{\mathbb Z}{S^\infty_n}}(X)$ by $\Psi(m)= (\iota, \widetilde{m})$, where $\iota : \Delta_{\text{sub}}^n \to {\mathbb A}^n$ is the inclusion. 
It is readily seen that $\Phi\circ \Psi = id$. Therefore the functor ${\mathbb Z}{S^\infty_n(\text{-})}$ is representable for 
${\mathcal M}$. Observe that the inclusion $\iota$ is in ${\mathbb Z}S^\infty_n({\mathbb A}^n)$. 

Since the identity maps $id_{{\mathbb A}^n}$ and  $id_{\Delta_{\text{sub}}^n}$ belong to 
${\mathbb Z}{S^D_n({\mathbb A}^n)}$ and ${\mathbb Z}{S^D_n(\Delta_{\text{sub}}^n)}_{\text{sub}}$, respectively, 
it follows from the same argument as above that the functors ${\mathbb Z}{S^D_n( \text{-})}$ and 
${\mathbb Z}{S^D_n( \text{-})}_{\text{sub}}$ are representable for ${\mathcal M}$. 
\end{proof}

We consider the excision axiom for the homology of 
${S^D_\bullet(X)}_{\text{sub}}= \mathsf{Diff}(\Delta_\text{sub}^n, X)$.  Let 
$D: \mathsf{Diff} \to \mathsf{Top}$ and $C:  \mathsf{Top} \to  \mathsf{Diff}$ be the functors mentioned in Appendix B, which give an adjoint pair.  
Kihara's proof of \cite[Proposition 3.1]{Kihara} enables us to regard the chain complex 
${\mathbb Z}{S^D_\bullet(X)}_{\text{sub}}$ as a subcomplex of the singular chain complex $C_*(DX)$, 
where $D: \mathsf{Diff} \to \mathsf{Top}$ denotes the functor mentioned in Appendix B.  In fact, \cite[Lemma 3.16]{C-S-W} implies that 
$D(\Delta_\text{sub}^n)$ is the simplex $\Delta^n$ with the standard topology. 
We observe that for the diffeology ${\mathbb R}^n$ with smooth plots, $D({\mathbb R}^n)$ is Euclidian space.  
Thus for a diffeological space $X$, the unit $id : X \to CDX$ yields the sequence of inclusions  
\[
\mathsf{Diff}(\Delta_\text{sub}^n, X) \to \mathsf{Diff}(\Delta_\text{sub}^n, CDX)  \cong 
\mathsf{Top}(D(\Delta_\text{sub}), DX) = \mathsf{Top}(\Delta^n, DX); 
\eqnlabel{add-5}
\]
see Remark \ref{rem:CDC} for a discussion of when the composite of the maps in (4.1) is bijective.  
Then we can prove the excision axiom by applying a barycentric subdivision argument. 
Indeed, the subdivision map 
$Sd : {S^D_n(X)}_\text{sub} \to {S^D_n(X)}_\text{sub}$ is defined by restricting the usual one for the singular chain complex, which is chain homotopic to the identity.  
It turns out that the relative homology $H_*^D(X, A)$ satisfies the excision axiom for the $D$-topology; that is, 
the inclusion $i : (X-U, A-U) \to (X, U)$ induces an isomorphism on the relative homology if the closure of $U$ is contained in the interior of $A$ with respect to 
the $D$-topology of $X$;  see \cite[IV, Section 17]{B} for details. 
Thus we have the Mayer--Vietoris exact sequences for the homology and cohomology of $S^D_\bullet(X)_{\text{sub}}$. 

In order to prove Theorem \ref{thm:main}, more observations concerning the cochain complexes in the theorem 
are now given. 

\medskip
\noindent
I) The method of acyclic models \cite[Section 8]{E-M} implies that there exists a chain homotopy equivalence 
$\widetilde{l} : {\mathbb Z}S^D_\bullet(X)_{\text{sub}} \stackrel{\simeq}{\longrightarrow} {C_\text{cube}}_*(X)$. 
The dual to $\widetilde{l}$ yields  
a cochain homotopy equivalence $l :C_\text{cube}^*(X) \stackrel{\simeq}{\longrightarrow} C^*(S^D_\bullet(X)_{\text{sub}})$, 
which induces a morphism of algebras on cohomology; see also \cite[Theorem 8.2]{SNPA} for details. 

\medskip
\noindent
II) The restriction map $j^* : {\mathbb Z}S^D_\bullet(X) \to {\mathbb Z}S^D_\bullet(X)_{\text{sub}}$ has a homotopy inverse $h$
in the category of chain complexes. This follows from the method of acyclic models \cite[Theorems Ia and Ib]{E-M} with 
Lemmas \ref{lem:acyclic} and \ref{lem:representable}. 
Then the map 
$h^* : C^*(S^D_\bullet(X)) \to C^*(S^D_\bullet(X)_{\text{sub}})$ induces an isomorphism of {\it algebras} on the homology. 
In fact, the inverse induced by 
$(j^*)^* : C^*(S^D_\bullet(X)_{\text{sub}}) \to C^*(S^D_\bullet(X))$ is a morphism of algebras. 

\begin{rem}\label{rem:M-V}
As seen in above, the homology $H^*(S^D_\bullet(X)_{\text{sub}})$ admits the Mayer--Vietoris exact sequence. Then, it follows from II) that so does $H^*(S^D_\bullet(X))$. 
\end{rem}

\noindent
III) The homotopy commutativity of the square in Theorem \ref{thm:main} also follows from the method of acyclic models for cochain complexes; 
see Appendix A for more details. 

\begin{proof}[Proof of the first assertion in Theorem \ref{thm:main}]
Proposition \ref{prop:3.4}, the considerations in  I), II), III) and the commutative diagram (3.2) allow us to deduce the first part. 
\end{proof}

\begin{proof}[Proof of Corollary \ref{cor:main}]
By Theorem \ref{thm:main}, we see that $\text{mult}\circ (1\otimes \int)$ is a quasi-isomorphism. 
The commutativity of the right triangle implies that the integration map is also a quasi-isomorphism.  
\end{proof}

\begin{proof}[Proof of Corollary \ref{cor:main2}] The argument in the beginning of Subsection \ref{sub4.1} gives (i). 
The assertion (ii) follows from II) above and the fact that the integration map $\int$ induces a morphism of algebras. 
The first assertion in Theorem \ref{thm:main} yields (iii). 
\end{proof}

\begin{proof}[Proof of the latter half of Theorem \ref{thm:main}]
We first observe that 
the Poincar\'e lemma for the cohomology of the de Rham complex in the sense of Souriau and the homotopy axiom hold 
for diffeological spaces; see \cite[Chapter 6]{IZ}.  
By Corollary \ref{cor:main}, 
it suffices to show that the composite $v:=\int \circ \ \alpha$ induces an isomorphism on the cohomology. 

For proving the result in case of a smooth CW-complex $\{K, \{K^{(n)}\}_{n\geq -1}\}$, 
we can use a Mayer--Vietoris exact sequence argument. 
To see this, we first observe that $K=\text{colim} \ K^{(n)}$ and the set $\{K^{(n)}\}_{n\geq -1}$ is finite by assumption. 
Under the same notation as in Appendix B, we have a partition of unity subordinate to the $D$-open subsets $U:= K^{(n)}- K^{(n-1)}$ and 
$V:= K^{(n)}-(\cup_{j\in J_n}\{0_j\})$, where $0_j$ denotes the origin of the ball $B_j$ attached to $K^{(n-1)}$; see \cite[Appendix]{I-I}. 
This yields the Mayer--Vietoris exact sequence of the de Rham cohomology functor $H^*(\Omega^*( \ ))$ with respect to $\{U, V\}$; see \cite[Theorem 4.3]{Haraguchi}. 
Moreover, we have seen that the cohomology functor $H^*(S^D_\bullet( \ )_{\text{sub}})$ admits the Mayer--Vietoris exact sequence for a $D$-open cover. 
The composite $v$ is natural with respect to smooth maps. Then the comparison of two exact sequences and an induction argument on the dimension 
of the skeleton $K^{(n)}$ yields the result for a smooth CW complex. 

Suppose that $(X, \D^X)$ is a manifold. Then the argument in \cite[V. \S9]{B} is applicable in order to deduce 
the map $v$ is a quasi-isomorphism. By definition, a parametrized stratifold $(S, \C)$ is constructed from manifolds 
with boundaries via an attaching procedure; see Appendix B. In general, a stratifold admits a partition of unity; see \cite[Proposition 2.3]{Kreck}. Moreover, we see that an open set of the underlying topological space $S$ is a $D$-open set of the diffeology $k(S, \C)$; see Lemma \ref{lem:D-top}. 
Thus the induction argument with the Mayer--Vietoris sequence works well 
to show that $H(v)$ is an isomorphism. In fact, let $S'$ be the parametrized stratifold mentioned in Appendix B, which is obtained from a stratifold $S$ and a manifold $W$ with collared boundary $\partial W$ by using an attaching map $f : \partial W \to S$. In the inductive step, we can use open sets of $S' = S\cup_f W$. 
$U := S\cup_f (\partial W\times [0, \varepsilon))$ and $V:= W- (\partial W \times [0, \varepsilon/2])$. Observe that $U$ is smoothly homotopy equivalent to $S$ and $V$ is a manifold without boundary. 
Hence we have the result for a parametrized stratifold.
\end{proof}

\subsection{An applications of the integration map in Theorem \ref{thm:main}}\label{sect6}

In this short section,  we describe an application of the integration map on $A^*_{DR}(S^D_\bullet(X))$ mentioned 
in Theorem \ref{thm:main}. 
Let $j^* : S_n^D(X) \to S_n^\infty(X)$ and $j^* : (A^*_{DR})_n \to (\widetilde{A^*_{DR}})_n$ be the restriction maps induced by 
the inclusion $j : \Delta_{\text{sub}}^n \to {\mathbb A}^{n}$. 
The naturality of the integration map $\int$ in the theorem implies that the map $\alpha'$ described in Section \ref{sect3.2} 
is an extension of $\alpha$ on the cohomology. 

\begin{prop}\label{prop:alpha_beta}
One has the diagram 
\[
\xymatrix@C40pt@R14pt{
                               & \mathsf{Sets^{\Delta^{op}}}(S_\bullet^\infty(X), (\widetilde{A^*_{DR}})_\bullet) & \hspace{-1.5cm}= \widetilde{A^*_{DR}}(S^\infty_\bullet(X)) \\
\Omega^* (X) \ar[ur]^{\alpha'} \ar[dr]_{\alpha}&  \mathsf{Sets^{\Delta^{op}}}(S_\bullet^\infty(X), (A^*_{DR})_\bullet) \ar[u]_{(j^*)_*}^{} \ar[d]^{(j^*)^*} & 
   \hspace{-1.5cm}= A^*_{DR}(S^\infty_\bullet(X))\\
          & \mathsf{Sets^{\Delta^{op}}}(S_\bullet^D(X), (A^*_{DR})_\bullet) & \hspace{-1.5cm}= A^*_{DR}(S^D_\bullet(X))
}
\]
in which $(j^*)_*$ and $(j^*)^*$ are quasi-isomorphisms. Moreover, the diagram is commutative on cohomology. 

\end{prop}

Observe that a natural map from $\Omega^*(X)$ to $A^*_{DR}(S^\infty_\bullet(X))$ cannot be defined 
in such a way as to give the map $\alpha$. 
However, 
Proposition \ref{prop:3.4} and the commutative diagram (3.2) yield that the integration $\int : (\widetilde{A_{DR}^*})_\bullet \to (C_{PL})_\bullet$ in 
Section \ref{sect3.2} gives rise to a quasi-isomorphism 
$\int_* :  \widetilde{A_{DR}^*}(K) \to  C_{PL}^*(K) \cong C^*(K)$
of differential graded modules for {\it each} simplicial set $K$, which is an isomorphism of algebras on the cohomology. 
This is a key to proving Proposition \ref{prop:alpha_beta}.  

\begin{proof}[Proof of Proposition \ref{prop:alpha_beta}] We consider the diagram 
\[
\xymatrix@C40pt@R14pt{
           & \mathsf{Sets^{\Delta^{op}}}(S_\bullet^\infty(X), (\widetilde{A^*_{DR}})_\bullet) \ar[rd]^{\int_*}_(0.4){\simeq} & 
                        \\
\Omega^* (X) \ar[ur]^{\alpha'} \ar[dr]_{\alpha}&  \mathsf{Sets^{\Delta^{op}}}(S_\bullet^\infty(X), (A^*_{DR})_\bullet) \ar[u]_{(j^*)_*}^{} \ar[d]^{(j^*)^*}  
\ar[r]^{\int_*}_{\simeq} & \mathsf{Sets^{\Delta^{op}}}(S_\bullet^\infty(X), C_{PL}) \ar[d]^{(j^*)^*}\\
          & \mathsf{Sets^{\Delta^{op}}}(S_\bullet^D(X), (A^*_{DR})_\bullet) \ar[r]_{\int_*}^{\simeq} &\mathsf{Sets^{\Delta^{op}}}(S_\bullet^D(X), C_{PL}).
}
\eqnlabel{add-20}
\]
By Lemmas \ref{lem:acyclic} and \ref{lem:representable}, we see that 
the method of acyclic models is applicable to the complexes  ${\mathbb Z}S_\bullet^D(X)$ and 
${\mathbb Z}S_\bullet^\infty(X)$. 
Then this yields that the restriction 
$(j^*) :  {\mathbb Z}S_\bullet^D(X) \to {\mathbb Z}S_\bullet^\infty(X)$ is a quasi-isomorphism and then so is 
the map $(j^*)^*$ on the right hand side. 
It follows that the center triangle and square are commutative by the definition of the integration map; see (3.1). 
The commutative diagram (3.2) implies that the integration maps are quasi-isomorphisms. 
Then we see that maps $(j^*)_*$ and $(j^*)^*$ on the left hand side are quasi-isomorphisms. 
Moreover,  a direct calculation indicates that $(j^*)^*\circ \int_* \circ \alpha' = \int_*\circ \alpha$. Therefore, the left triangle is commutative on cohomology. This completes the proof.   
\end{proof}


\begin{rem}
Let $(S, \C)$ be a parametrized stratifold; see Appendix B. By virtue of Lemma \ref{lem:D-top}, it is readily seen that 
the identity map on the set $S$ gives rise to the map $i : D(S) \to S$ in $\mathsf{Top}$.  
Consider the map $\iota : H_*(S^D_\bullet(k(S, \C))) \to H_*(D(S)) \to H_*(S)$ to the singular homology with coefficients in ${\mathbb Z}$, which is induced by 
the map $i$ and the inclusion mentioned in (4.1); see Appendix B for the functor $k$.

As seen in the proof of Proposition \ref{prop:alpha_beta}, the restriction map 
$(j^*) :  {\mathbb Z}S_\bullet^D(X) \to {\mathbb Z}S_\bullet^\infty(X)$ is a quasi-isomorphism. Moreover, 
the map $\iota$ is nothing but the composite $\iota'\circ H(j^*)$, where 
$\iota' : H_*(S^\infty_\bullet(k(S, \C))) \to H_*(D(S)) \to H_*(S)$ induced by the map $i$ 
and the restriction of the inclusion mentioned in (4.1) of $S_n^\infty(k(S, \C)))\subset S_n^D(k(S, \C)))=\mathsf{Diff}(\Delta_{\text{sub}}^n, k(S, \C))$. 
If $S$ is a manifold, then $\iota'$ is an isomorphism. 
This follows from the argument in the proof of \cite[V. 9.5 Lemma]{B}, which asserts that the smooth singular homology is isomorphic to the ordinary singular homology. By Remark \ref{rem:M-V}, we have the Mayer--Vietoris exact sequence for the homology $H_*(S^D_\bullet(X)))$ and hence for $H_*(S^\infty_\bullet(X)))$ of a diffeological space $X$. Then the argument with the Mayer--Vietoris exact sequence as in the proof of Theorem \ref{thm:main} enables us to conclude that 
$\iota$ is an isomorphism for every parametrized stratifold. 
\end{rem}


\section{Chen's iterated integral map in diffeology}\label{sect7} 
We begin by modifying the iterated integrals due to Chen \cite{C} in the diffeological setting.  
Let $N$ be a diffeological space and $\rho : \R \to I$ a cut-off function with $\rho(0)=0$ and $\rho(1)=1$. 
Then a $p$-form $u$ on the diffeological space $I\times N$ is called an $\Omega^p(N)$-{\it valued function on} $I$ if 
for any plot $\psi : U \to N$ of $N$, the $p$-form $u_{\rho \times \psi}$ 
on $\R \times U$ is of the type 
\[
\sum a_{i_1\cdots i_p}(t, \xi)d\xi_{i_1}\wedge \cdots \wedge d\xi_{i_p},
\]
where $(\xi_1, ..., \xi_n)$ denotes the coordinates of $U$ we fix. For such an $\Omega^p(N)$-valued function $u$ on $I$, we define the integration 
$\int_0^1u \ dt \in \Omega^p(N)$ by
\[
(\int_0^1u \ dt)_\psi = \sum (\int_0^1a_{i_1\cdots i_p}(t, \xi) \ dt) d\xi_{i_1}\wedge \cdots \wedge d\xi_{i_p}. 
\]
Each $p$-form $u$ has the form $u = dt\wedge ((\partial /\partial t) \rfloor u) + u''$, where $(\partial /\partial t) \rfloor u$ and $u''$ are an 
$\Omega^{p-1}(N)$-valued function and an $\Omega^{p}(N)$-valued function on $I$, respectively. Let $F : I \times N^I \to N^I$ be the homotopy defined by 
$F(t, \gamma)(s) = \gamma(ts)$. The Poincar\'e operator $\int_F : \Omega(N^I) \to \Omega(N^I)$ associated with the homotopy $F$ is defined by 
$\int_F v = \int_0^1((\partial /\partial t) \rfloor F^*v)dt$. 
Moreover, for forms $\omega_1$, ..., $\omega_r$ on $N$, 
the {\it iterated integral} $\int \omega_1\cdots \omega_r$ is defined by $\int \omega_1 = \int_F\e_1^*\omega_1$ and 
\[
\int \omega_1\cdots \omega_r = \int_F\{J(\int \omega_1\cdots \omega_{r-1}) \wedge \e_1^*\omega_r\},
\]
where $\e_i$ denotes the evaluation map at $i$, $Ju =(-1)^{\deg u}u$ and $\int \omega_1\cdots \omega_r =1$ if $r=0$; see \cite[Definition 1.5.1]{C}. 
Observe that the Poincar\'e operator is of degree $-1$ and then $\int \omega_1\cdots \omega_r$ is of degree $\sum_{1\leq i \leq r}(\deg \omega_i -1)$. 

With a decomposition of the form $A^1\oplus d\Omega^0(N)$, 
we obtain a cochain subalgebra $A$ of $\Omega(N)$ which satisfies the condition that $A^p = \Omega(N)$ for $p> 1$ and $A^0=\R$. The cochain algebra $A$ gives rise to the normalized bar complex $B(\Omega(N), A, \Omega(N))$; see \cite[\S 4.1]{C}. Consider the pullback diagram 
\[
\xymatrix@C25pt@R20pt{
E_f \ar[r]^-{\widetilde{f}} \ar[d]_{p_f} & N^I \ar[d]^{(\e_0, \e_1)}\\ 
M \ar[r]_-{f} & N\times N
}
\eqnlabel{add-6}
\]
of $(\e_0, \e_1) : N^I \to N\times N$ along a smooth map $f : M \to N\times N$. 
We write $\overline{B}(A)$ for 
$B(\R, A, \R)$. 
Then we have a map  
\[
\mathsf{It} : \Omega(M)\otimes_{\Omega(N)\otimes\Omega(N)}B(\Omega(N), A, \Omega(N))\cong \Omega(M) \otimes_f \overline{B}(A) \to \Omega(E_f)
\]
defined by 
$\mathsf{It} (v\otimes [\omega_1| \cdots | \omega_r])=
p_f^*v\wedge \widetilde{f}^*\int \omega_1\cdots \omega_r$. 
Observe that the source of $\mathsf{It}$ gives rise to the differential on 
$\Omega(M) \otimes_f \overline{B}(A)$ by definition. Since $\rho(0)=0$ and $\rho(1)=1$ for the cut-off function $\rho$ which we use when defining the $\Omega^p(N)$-valued function on $I$, it follows that the result 
\cite[Lemma 1.4.1]{C} remains valid.  
Then the formula of iterated integrals with respect to the differential in \cite[Proposition 1.5.2]{C} implies that 
$\mathsf{It}$ is a well-defined morphism of differential graded $\Omega^*(M)$-modules.

\begin{rem}
The cut-off function $\rho$ does not satisfy the formula $\rho(s)\rho(t) = \rho(st)$ for $s, t \in \R$ in general. Then we do not have the same assertion 
as that of  \cite[Lemma 1.5.1]{C} which allows us to deduce a constructive definition of the iterated integral as in (2.2); see \cite[page 840]{C}.
\end{rem}

\begin{thm}\label{thm:general_main} Suppose that, in the pullback diagram \text{\em (5.1)}, the diffeological space $N$ is 
simply connected and $f$ is an induction.  
Assume further that the factor map $\alpha : \Omega^*(M) \to A_{DR}(S^D_\bullet(M))$ 
is a quasi-isomorphism and that the cohomology $H^*(S^D_\bullet(M))$ is of finite type, that is, each vector space $H^i(S^D_\bullet(M))$ is of finite dimension. Then the composite 
$\alpha \circ  \mathsf{It} : \Omega^*(M)\otimes_f \overline{B}(A) \to 
\Omega(E_f) \to A^*_{DR}(S^D_\bullet(E_f))$ is a quasi-isomorphism of 
$\Omega^*(M)$-modules.  
\end{thm}

Let $g: (X, \D^X) \to (Y, \D^Y)$ be an induction. Then by definition, the map $g$ is injective and the pullback diffeology $g^*(\D^Y)$ coincides with $\D^X$. 
Then the result \cite[I. 36]{IZ} yields that the map $g : X \stackrel{\cong}{\to} g(X)$ is a diffeomorphism, where 
$g(X)$ is the diffeological space endowed with subdiffeology.

\begin{proof}[Proof of Theorem \ref{thm:the_second_main}]
The diagonal map $\Delta : M \to M\times M$ is an induction. Then the result follows from Theorem \ref{thm:general_main}. 
\end{proof}

\begin{rem}\label{rem:highlight}
Theorem \ref{thm:general_main} is regarded as a generalization of the result \cite[Theorem 0.1]{C2} in which $M$ and $N$ are assumed to be manifolds. 
The theorem due to Chen asserts that the homology of the bar complex 
$\Omega^*(M)\otimes_f \overline{B}(A)$ is isomorphic to the cubical cohomology of the Chen space $E_f$ via the pairing with the iterated integrals and cubic smooth chains; see \cite[(2.2)]{C2}.  
\end{rem}

The rest of this section is devoted to proving Theorem \ref{thm:general_main}. 
Recall the definition of a local system over a simplicial set and its global sections.  Let $K$ be a simplicial set. 
We regard $K$ as a category whose objects are simplicial maps $\sigma : \Delta[p] \to K$ for $p \geq 0$ and whose  morphisms $\alpha : \tau \to \sigma$ are simplicial maps  $\alpha : \Delta[q] \to \Delta[p]$ 
with  $\tau = \sigma \circ \alpha$, 
where $\tau : \Delta[q] \to K$ and $\sigma : \Delta[p] \to K$. 
Then a {\it local system $F$ over $K$ of differential coefficients} is defined to be a contravariant functor from $K$ to the category $\mathsf{coChAlg}$ of unital cochain algebras with non-negative grading which satisfies the condition that the map $F(\alpha) : F_{\sigma}\to 
F_{\alpha^*\sigma}$ is a quasi-isomorphism for each $\alpha$ in the category $K$; see \cite[Definition 12.15]{H}. Observe that such a local system 
$F$ is an object of the functor category $\mathcal{E} := \mathsf{coChAlg}^{K^{\text{op}}}$. 

Forgetting the cochain algebra structure of $F$, the functor is regarded as an object of the functor category 
$U\mathcal{E} := \mathsf{Vect}^{K^{\text{op}}}$ to the category of vector spaces over ${\mathbb R}$.  
We define the space 
$\Gamma(F)$ of global sections of $F$ by $\Gamma(F):= \text{Hom}_{U\mathcal{E}}({\mathbb R}, F)$ 
which is isomorphic to $\text{Hom}_{\mathsf{Set}^{K^{\text{op}}}}(1, F)$, where $1$ denotes the terminal object of 
$\mathsf{Set}^{K^{\text{op}}}$ the category  of presheaves.  
We observe that $\Gamma(F)$ is a cochain algebra whose structure is defined pointwise, for example, $(d\phi)_\sigma := (d\phi_\sigma)$ for a simplex $\sigma$ in $K$. We refer the reader to \cite[Chapter 12]{H} for details of such local systems. 

There are at least two kinds of {\it fibrations} in the category $\mathsf{Diff}$. One of them is a fibration 
$f :  X \to Y$ in the sense of Christensen and Wu \cite{C-W}, namely a smooth map which induces a fibration  
$S^D(f) : S^D_\bullet(X) \to S^D_\bullet(Y)$ of simplicial sets in $\mathsf{Sets}^{\Delta^{op}}$.
An important example of such a fibration is a diffeological bundle in the sense of  Iglesias-Zemmour \cite[Chapter 8]{IZ} whose fibre is fibrant;  
see \cite[Proposition 4.24]{C-W}. 

Another type concerns mapping spaces with evaluation maps. For example, with the interval $I=[0,1]$, 
the map $(\e_0, \e_1) : N^I \to N\times N$ defined by the evaluation map $\e_i$ at $i$ is a fibration in $\mathsf{Top}$ if $N^I$ is endowed with compact open topology; that is, the map $(\e_0, \e_1)$ enjoys the right lifting property with respect to the inclusion $\Delta^n \to \Delta^n\times \{0\} \to \Delta^n\times I$ for 
$n \geq 0$. However, it seems that a smooth lifting problem is not solved in general for such an evaluation map in $\mathsf{Diff}$. 
Then in order to prove Theorem \ref{thm:general_main}, we reconstruct the Leray--Serre spectral sequence and algebraic models 
for path spaces in the diffeological framework. 

In what follows, we may write  $A^*_{DR}(X)$ and $A^*(X)$ for  $A^*_{DR}(S^D_\bullet(X)_{\text{sub}})$ and 
$A^*_{DR}(S^D_\bullet(X))$, respectively. Observe that 
the map $(j^*) :  S^D_\bullet(X) \to S^D_\bullet(X)_{\text{sub}}$ 
induced by inclusion $j : \Delta^n_{\text{sub}} \to  {\mathbb A}^{n}$ 
gives rise to a natural quasi-isomorphism  
\[
(j^*)^* : A^*_{DR}(X) \to A^*(X).
\eqnlabel{add-7}
\] This follows from II) in Section \ref{sub4.1}, Proposition \ref{prop:3.4} and the naturality of the integration map; 
see the proof of Proposition \ref{prop:alpha_beta}.
The argument in \cite[Sections 5, 6 and 7]{Grivel} due to Grivel enables us to obtain the Leray--Serre spectral sequence with a local system for a fibration and the Eilenberg--Moore spectral sequence for a fibre square. 

\begin{thm} \label{thm:LSSS}
Let $\pi : E \to M$ be a smooth map between path-connected diffeological spaces with path-connected fibre $L$ which is 
\text{\em i)} a fibration in the sense of Christensen and Wu or \text{\em ii)} 
the pullback of the evaluation map $(\e_0, \e_1) : N^I \to N\times N$ for a connected diffeological space $N$ along an induction 
$f :  M \to N\times N$.  Suppose further that in the case \text{\em ii)} the cohomology $H(A^*(M))$ is of finite type. Then one has the Leary--Serre spectral sequence 
$\{_{LS}E_r^{*,*}, d_r\}$ converging to 
$H(A^*(E))$ as an algebra with an isomorphism 
\[
_{LS}E_2^{*,*}\cong H^*(M, \mathcal{H}^*(L))
\]
of  bigraded algebras, 
where $H^*(M, \mathcal{H}(L))$ is the cohomology with the local coefficients $\mathcal{H}^*(L)=\{H(A^*(L_c))\}_{c\in S^D_0(M)}$; see Lemma \ref{lem:DiffCoeff} below. 
\end{thm}

\begin{thm} \label{thm:EMSS}
Let $\pi : E \to M$ be the smooth map as in Theorem \ref{thm:LSSS} with the same assumption, 
$\varphi : X \to M$ a smooth map from a connected diffeological space $X$ for which the cohomology $H(A^*(X))$ is of finite type and 
$E_\varphi$ the pullback of $\pi$ along $\varphi$.  Suppose further that $M$ is simply connected in case of 
\text{\em i)} and $N$ is simply connected in case of \text{\em ii)}. Then one has the Eilenberg--Moore spectral sequence 
$\{_{EM}E_r^{*,*}, d_r\} $ converging to 
$H(A^*(E_\varphi))$ as an algebra with an isomorphism 
\[
_{EM}E_2^{*,*} \cong \text{\em Tor}_{H(A^*(M))}^{*,*}(H(A^*(X)), H(A^*(E)))
\]
of bigraded algebras. 
\end{thm}

\begin{proof}[Proofs of Theorems \ref{thm:LSSS} and \ref{thm:EMSS}]
For the case i), the Leray--Serre spectral sequence and the Eilenberg--Moore spectral sequence 
are obtained by applying the same argument as in the proofs of  
\cite[5.1 Theorem and 7.3 Theorem]{Grivel} to the functor $A^*( \ ):=A^*_{DR}(S^D( \ )_\bullet)$. Observe that 
Dress' construction for the Leary-Serre spectral sequence is applicable to our setting; see \cite[3.3]{Grivel} and \cite{McCleary}.  

To consider the case ii), we use $A^*_{DR}( \ )$ instead of $A^*(\ )$. Then, the spectral sequence $\{'E_r^{*,*}, 'd_r\}$ constructed in the proof of 
Proposition  \ref{prop:KSextension} below gives the one in Theorem \ref{thm:LSSS}. 
As for the Eilenberg--Moore spectral sequence, 
by virtue of the result \cite[20.6]{H} and Proposition \ref{prop:KSextension} below, we have $H^*(A_{DR}(E_f)) \cong 
\text{Tor}_{A_{DR}^*(M)}(A_{DR}^*(X), A_{DR}^*(E))$ as an algebra; see \cite[Th\'eor\`eme 4.1.1]{VP}.  
As a consequence, the natural quasi-isomorphism $(j^*)^*$ in (5.2) yields the result in Theorem  \ref{thm:EMSS}. 
\end{proof}

\begin{rem}
In Theorems  \ref{thm:LSSS} and \ref{thm:EMSS}, we have dealt with {\it fibrations} of the type i) and of the type ii). We do not know whether the second notion of fibration coincides with the first notion. Therefore, we have considered the two cases separately. 
%
\end{rem}

To prove Theorem \ref{thm:general_main}, the argument in 
\cite[Chapter 19]{H} will be replaced with that in our setting.  Recall the standard face and degeneracy maps 
$\eta_i : \Delta_{\text{sub}}^{p-1} \to \Delta_{\text{sub}}^{p}$ and $\zeta_j : \Delta_{\text{sub}}^{p+1} \to \Delta_{\text{sub}}^{p}$. 
For $0 \leq m \leq p$, let $\alpha_m :   \Delta_{\text{sub}}^{p+1} \to  \Delta_{\text{sub}}^{p} \times I$ be a smooth map defined by 
$\alpha_m (x) =(\zeta_m(x), \sum_{i=m+1}^{p+1}x_i)$, where $x =(x_0, x_1, ..., x_{p+1})$. Observe that the maps $\alpha_m$ give the standard triangulation of $\Delta^{p}\times I$ in the category of topological spaces. For any nondecreasing map $u : [p] \to [n]$, we denote by the same notation the affine map  $\Delta_{\text{sub}}^{p} \to  \Delta_{\text{sub}}^{n}$ defined by $u$. Such an affine map $u$ gives a set map 
$\overline{u} : \Delta^{p}\times I \to \Delta^{n}$ defined by 
$(\overline{u} \alpha_m)(\sum_{i=0}^{p+1} \lambda_i v_i) = \sum_{i=0}^{m} \lambda_i v_{u(i)} + (\sum_{i=m+1}^{p+1}x_i)v_n$; see \cite[(19.4)]{H}. It is readily seen that the composite $\overline{u} \alpha_m$ is a smooth map for each $m$.

Let $N$ be a diffeological space and $f : M \to N\times N$ an induction.  Then the pullback $\nu' : E_f \to M$ of the map $(\e_0, \e_1) : N^I \to N\times N$ along $i$ 
is identified with the  map 
\[
\nu := f^{-1}\circ (\e_0, \e_1) : P_MN:=\{\gamma \in N^I \mid (\e_0, \e_1)(\gamma) \in f(M)\} \to M. 
\]  
We consider the homotopy pullback $\pi : P_\sigma^h \to \Delta_{\text{sub}}^{n}$ of $\nu$ along the $n$-simplex 
$\sigma : \Delta_{\text{sub}}^{n} \to M$. By definition, it is given by  
\[
P_\sigma^h := \{(a, \zeta, \gamma) \in \Delta_{\text{sub}}^{n}\times M^I\times P_MN \mid \sigma(a) = \zeta(0), \ \nu(\gamma) = \zeta(1)\} 
\] 
with $\pi$ the projection. Let $\underline{P_\sigma^h}$ be the sub-simplicial set of 
$S^D_\bullet(P_\sigma^h)_{\text{sub}}$ consisting of  $p$-simplices each of which satisfies the condition that $\pi\circ \sigma : \Delta^{p} \to \Delta^{n}$ is the affine map defined by a nondecreasing map $[p] \to [n]$.  
By the same way,  we have a sub-simplicial set $\underline{P_\sigma}$ of 
$S^D_\bullet(P_\sigma)_{\text{sub}}$, where $P_\sigma$ denotes the pullback of $\nu$ along the $n$-simplex 
$\sigma : \Delta_{\text{sub}}^{n} \to M$. 
The restriction map of $\sigma$ to the vertex $\sigma(n)$ gives rise to the pullback $P_{\sigma(n)}$ of $\nu : P_MN \to M$. Then we have the natural inclusion 
$j : P_{\sigma(n)} \to P_\sigma^h$ defined by $j(\gamma) = (\sigma(n), C_{\sigma(n)}, \gamma)$, where $C_{\sigma(n)}$ is the constant map at 
$\sigma(n)$. The following lemma is proved by modifying the argument of the proof of \cite[Lemma 19.9]{H} in the diffeological framework. 

\begin{lem}\label{lem:KEY}
The homomorphism 
$
\xymatrix@C25pt@R15pt{
A^*_{DR}(P_{\sigma(n)}) & 
\ar[l]_-{A(j)} A^*_{DR}(\underline{P_\sigma^h}) 
}
$
induced by the natural map  $j : S^D_\bullet(P_{\sigma(n)})_{\text{\em sub}} \to \underline{P_\sigma^h}$ 
of simplicial sets is a quasi-isomorphism. 
\end{lem}

\begin{proof} By Theorem \ref{thm:main}, it suffices to show that the map $j_* : C_*( S^D_\bullet(P_{\sigma(n)})) \to C_*(\underline{P_\sigma^h})$ of chain complexes induced by $j$ is a chain homotopy equivalence. We identify $M$ with the subspace $f(M)$ of $N\times N$. 
For each $\tau \in( \underline{P_\sigma^h})_p$, we have the map 
$\overline{\pi\circ \tau} : \Delta^{p} \times I \to \Delta^{n}$ mentioned above. Let $J$ denote the space $(I \times \{0, 1\}) \cup (\{0\} \times I)$ 
endowed with the subdiffeology of $I\times I$. 
For each point $z \in I\times I$, join $(2, \frac{1}{2})$ to $z$ 
by a straight line, and make the line beyond $z$ until it meets $J$ at a point $z'$. 
Then one defines a retraction $r : I\times I \to J$ by $r(z)=z'$ as a set map. 
Moreover, by using the projections $\text{pr}_i$ from $P_\sigma^h$ in the $i$th factor and the adjoint $ad$ of a given map to $M^I$,  
we define a map $\widetilde{\tau} : \Delta^{p}\times  I \times I \to M$ by 
\[
\widetilde{\tau} := ((\sigma\circ \overline{\pi\circ \tau})_0\circ(1\times \rho)) \cup  ((\sigma\circ \overline{\pi\circ \tau})_1\circ (1\times \rho)) \cup (
\omega^{-1} \ast \ell \ast \omega) \circ (1\times r), 
\]
where $\omega(z, 0, s)=ad(\text{pr}_2\tau)(z, \rho(s))$, $(\omega)^{-1}(z, 0, s)=ad(\text{pr}_2\tau)(z, 1- \rho(s))$, 
$\ell(z, 0, s)=ad(\text{pr}_3\tau)(z, \rho(s))$ and $\rho$ is a cut-off function.  Moreover, 
$(\sigma\circ \overline{\pi\circ \tau})_i$ denotes the map defined by $(\sigma\circ \overline{\pi\circ \tau})$ on 
$\Delta^{p}\times I \times \{ i \}$ for $i = 0, 1$.  
We observe that $\widetilde{\tau}$ is smooth on $(\text{Im}\ \alpha_m) \times I$. In fact, 
the map is constant in appropriate neighborhoods of the rays from $(2, \frac{1}{2})$ to the points $(0,0 )$, $(0,1)$, $(0, \frac{1}{3})$ and $(0, \frac{2}{3})$ in $I\times I$. 
For any plot $p :  U \to \Delta^{p}\times  I \times I$, we take a point $r \in U$ and write $p(r) = (x, (a,b))$. 
Then there exists an open neighborhood $U$ of $(a, b)$ such that  $\widetilde{\tau}\circ p |_A$ is  constant or  
a composite of $(1\times r)$ and $(\sigma\circ \overline{\pi\circ \tau})_i$ or $(\omega^{-1} \ast \ell \ast \omega)$, 
where $A = p^{-1}(\Delta^{p}\times U)$. This implies that  $\widetilde{\tau}\circ p |_A$ is locally smooth. It follows 
that $\widetilde{\tau}\circ p$ is in $\D^M$ the diffeology of $M$ and then $\widetilde{\tau}$ is smooth.

We define $\overline{\overline{\tau}}' : \Delta^{p}\times  I  \to \text{Map}(I, M)$ by the adjoint to $\overline{\tau}$. The homotopy 
$H_\tau : (\Delta^{p}\times  I) \times I \to M$ from $\sigma\circ (\overline{\pi\circ \tau})$ to $\nu \circ \overline{\overline{\tau}}'$ defined by 
$H_\tau((z, t)), s) = (\sigma\circ \overline{\pi\circ \tau})(z,(1-s)t + s\rho(t))$ 
gives a map  
 $\overline{\overline{\tau}} : \Delta^{p}\times  I \to P_\sigma^h$ with 
\[
\overline{\overline{\tau}}(z,t) =( (\overline{\pi\circ \tau})(z,t), ad(H_\tau)(z,t), \overline{\overline{\tau}}'(z, t)). 
\]
Observe that the domain of the map will be restricted to the space $(\text{Im}\ \alpha_m)$ when constructing a simplicial homotopy below. We call $\overline{\overline{\tau}}$ 
the {\it canonical lift} associated with $\tau$. 
\[
\xymatrix@C35pt@R1pt{
\bullet \ar@{-}[rr]^-{(\sigma\circ \overline{\pi\circ \tau})_1\circ (1\times \rho))} \ar@{-}[dd]_{\omega^{-1}} & & \bullet \ar@{.}[dddddd] &&\\
   & &  & &\\
\bullet \ar@{-}[dd]_-{\ell} && && \\
   & && & \bullet (2, \frac{1}{2}) \ar[lllluuu] \ar[lllld]\\
\bullet \ar@{-}[dd]_-{\omega} && && \\
     & & & & \\
\bullet \ar@{-}[rr]_-{(\sigma\circ \overline{\pi\circ \tau})_0\circ (1\times \rho))}  & & \bullet  && 
}
\eqnlabel{add-8}
\]
Since $\overline{\pi\circ \tau}(z, 1)$ is the constant map at $v_n$, it follows that $\overline{\overline{\tau}}( \text{-}, 1)$ factors through $P_{v_n}$. 
Then we define a simplicial map 
$\lambda : \underline{P_\sigma^h} \to S^D_\bullet(P_{\sigma(n)})_{\text{sub}}$ by $\lambda(\tau) =  \overline{\overline{\tau}}(\text{-}, 1)$. 

In order to show that $j \circ \lambda$ is homotopic to the identity,  
we define $h_m : (\underline{P_\sigma^h})_p \to (\underline{P_\sigma^h})_{p+1}$ by $h_m(\tau)= \overline{\overline{\tau}} \circ \alpha_m$ for any 
$0\leq m \leq p$ by using the canonical lift. Since $\pi h_m(\tau)(z) = \overline{\pi\circ \tau} \alpha_m(z)$, it follows that $h_m$ is well defined. 
Then we see that $\{h_m\}_{0\leq m\leq p}$ gives rise to a simplicial homotopy; that is, the maps $h_m$ satisfy the following equalities 
\[
d_jh_m = 
\begin{cases}
f & \text{if $j=m=0$,} \\
h_{m-1}d_j & \text{if $j < m$,}   \\
d_jh_{m+1} & \text{if $0\leq j-1  =m < p$,} \\
h_md_{j-1}  & \text{if $0\leq m  < j-1 \leq p$,} \\
g & \text{if $j-1=m=p$,}
\end{cases}
\eqnlabel{add-9}
\]
where $f(\tau)=\overline{\overline{\tau}}(\text{-}, 1)$ and $g(\tau)=\overline{\overline{\tau}}(\text{-}, 0)$; see \cite{Barr} for details.
The construction of the canonical lift yields that $d_jh_m(\tau) = \overline{\overline{\tau}}\circ \alpha_m\circ \eta_j$ and 
$h_md_j(\tau) = \overline{\overline{\tau}}\circ (\eta_j\times 1)\circ \alpha_m$. Therefore, each equality mentioned above follows from the the relations between $\alpha_m$ and $\eta_m$. 

It remains to verify that $\overline{\overline{\tau}}(\text{-}, 0)$ is homotopic to the identity. Let $\tau$ be an element  in $(\underline{P_\sigma^h})_p$. We write  
$\tau (z) = (\pi\circ \tau(z), \zeta(z), \ell(z))$ and $\overline{\overline{\tau}}(z, 0) = (\overline{\overline{\tau}}(z, 0), C_y, \ell'(z))$, where $z \in \Delta^n$ and 
$C_y$ denotes the constant path at $y = \overline{\pi\circ \tau}(z, 0)$. Observe that $\overline{\pi\circ \tau}(z, 0) = (\pi\circ \tau)(z)$ and $\ell'$ is nothing but the path $(\omega^{-1}) \ast \ell \ast \omega$ in (5.3). Thus by using homotopies from $C_y$ to $\omega(z)$, 
from $(\omega^{-1}) \ast \ell \ast \omega (z, \text{-})$ to $\ell(z)$ and from the cut-off function $\rho$ to the identity, we can construct a homotopy 
$H_\tau : \Delta_{\text{sub}}^p \times I \to P_\sigma^h$ from $\overline{\overline{\tau}}(\text{-}, 0)$  to $\tau$ with $H_{d_j\tau} = H_\tau \circ (\eta_j \times 1)$. Define an element $h_m(\tau)$ in 
$(\underline{P_\sigma^h})_{p+1}$ by $h_m(\tau) = H_\tau\circ \alpha_m$. By a direct calculation, we see that $h_m$ satisfies the equalities in (5.4), where 
$f(\tau) = \tau$ and $g(\tau) = \overline{\overline{\tau}}(\text{-}, 0)$.

As for the composite 
$\lambda \circ j : S^D_\bullet (P_{\sigma(n)}) \to \underline{P_\sigma^h} \to S^D_\bullet(P_{\sigma(n)})$, it is also homotopic to the identity. 
The same homotopy as $H_\tau$ mentioned above gives such a chain homotopy. 
\end{proof}

\begin{rem}
We define a map $\varphi : M^I \to \mathsf{stPath}_\varepsilon(M)$ by composing a cut-off function $\rho$ with 
$\rho(0)=0$ and $\rho(1) =1$, where the target denotes the stationary path space; see \cite[5.4]{IZ}. Then the map $\varphi$ is smooth.  This follows from the smoothness of the evaluation map. Moreover, it follows from the locality of plots that the concatenation 
$\mathsf{stPath}_\varepsilon(M) \times \mathsf{stPath}_\varepsilon(M) \to \mathsf{stPath}_\varepsilon(M)$ is also smooth. 
By using these facts, we have proved Lemma \ref{lem:KEY}.  
\end{rem}

\begin{lem} \label{lem:KEY2} 
The map  $S^D(\iota) : \underline{P_\sigma} \to \underline{P_\sigma^h}$ defined by the inclusion 
$\iota : P_\sigma \to P_\sigma^h$ 
induces a quasi-isomorphism $\iota^* : A^*_{DR}(\underline{P_\sigma^h})  \to A^*_{DR}(\underline{P_\sigma})$. 
\end{lem}

\begin{proof}
The inclusion $\iota$ is given by $\iota(a, \gamma) = (a, C_{\sigma(a), \gamma})$.  We define a map 
$\mu :  P_\sigma^h \to P_\sigma$ by $\mu(a, \omega, \gamma) 
= (a, \widetilde{\omega^{-1}\ast \gamma \ast \omega})$, where 
$\widetilde{\omega^{-1}\ast \gamma \ast \omega} = 
(\omega^{-1}\circ \rho)\ast (\gamma\circ \rho)\ast (\omega\circ \rho)$. By adjusting the parameters of paths $\omega$ and $\gamma$,  
we can construct smooth homotopies $H : P_\sigma^h \times I \to P_\sigma^h$ from $1$ to $\iota\circ \mu$ and 
$G : P_\sigma \times I \to P_\sigma$ from $1$ to $\mu \circ \iota$ which preserve the first factor. For an $n$-simplex $\tau : \Delta^n \to P_\sigma^h$, define 
$h_m(\tau)$ by the composite $H\circ(\tau\times 1)\circ \alpha_m : \Delta^{n+1} \to \Delta^n \times I \to P_\sigma^h \times I \to P_\sigma^h$. Then the same argument as in the proof of Lemma \ref{lem:KEY} yields that the family 
$\{ \{h_m\}_{0\leq m\leq n} \}_{n\geq 0}$ gives a simplicial homotopy on $\{S^D_\bullet(\underline{P_\sigma^h)}\}$. By using the homotopy $G$, we have a simplicial homotopy on $\{S^D_\bullet(\underline{P_\sigma)}\}$. This completes the proof. 
\end{proof}

Lemmas \ref{lem:KEY} and \ref{lem:KEY2} imply that the map $\text{top}_n : [0] \to [n]$ defined by $\text{top}_n(0) =(n)$ 
induces a quasi-isomorphism 
$\text{top}_n^* : A^*_{DR}(\underline{P_\sigma})  \to A^*_{DR}(\underline{P_{\sigma(n)}})$. 
The map $\alpha(\eta) : \Delta_{\text{sub}}^{m} \to \Delta_{\text{sub}}^{n}$ induced by a nondecreasing map $\eta : [m] \to [n]$ fits into the commutative diagram 
\[
\xymatrix@C30pt@R18pt{
P_{\alpha(\eta)^*\sigma}  \ar[r]^-{\xi_{\alpha(\eta)}} \ar[d]_{\pi_{{\alpha(\eta)}^*\sigma}} &  P_\sigma \ar[d]^{\pi_\sigma} \ar[r]^{\xi_\sigma}& 
       E_f \ar[d]^{\nu} \\
 \Delta_{\text{sub}}^m \ar[r]_-{\alpha(\eta)} & \Delta_{\text{sub}}^n \ar[r]_\sigma & M. 
}
\]
Moreover, the smooth map $\xi_{\alpha(\eta)}$ gives rise to 
a simplicial map $\underline{\xi_{\alpha(\eta)}} :  \underline{P_{{\alpha(\eta)}^*\sigma}} \to \underline{P_\sigma}$. 
For the rest of this section, let $K$ denote the simplicial set ${S^D_\bullet(M)}_{\text{\em sub}}$. 
By using the quasi-isomorphisms $\text{top}_n^*$ mentioned above, we have 

\begin{lem}\label{lem:DiffCoeff} 
The family  $F=\{F_\sigma\}_{\sigma \in K}:=\{ A^*_{DR}(\underline{P_\sigma}) \}_{\sigma \in K}$ of cochain algebras gives an extendable local system over $K$ of differential coefficients.  
\end{lem}
\begin{proof}
For a nondecreasing map $\eta : [m] \to [n]$, we define a map $\tau : [n] \to [n+1]$ which sends $\eta(m), \eta(m)+1, ..., n$ to $n+1$. 
Since $\tau \circ \text{top}_n = \text{top}_{n+1}$ and $\tau\circ \eta \circ \text{top}_m = \text{top}_{n+1}$, it follows that  
$\underline{\xi_{\alpha(\eta)}}$ induces a quasi-isomorphism 
$(\underline{\xi_{\alpha(\eta)}})^* :  A_{DR}(\underline{P_\sigma}) \to A_{DR}(\underline{P_{{\alpha(\eta)}^*\sigma}})$. The extendability of the local system follows from Lemma \ref{lem:extendability} and the proof of \cite[19.17 Lemma]{H}. 
\end{proof}

\noindent
Let $j : R:=A_{DR}(M)\otimes \wedge V \to A_{DR}(E_f)$ be a KS extension for the map $\nu^* :  A_{DR}(M) \to A_{DR}(E_f)$ induced by the projection $\nu : E_f \to M$. The reader is referred to \cite[Chapter 1]{H} for the definition of a KS extension and its fundamental properties. 
Let $P_m$  denote the fibre  over a point $m \in M$. Since the composite of $\nu$ and the inclusion $l : P_m \to E_f$ is the constant map at $m$, 
it follows that the map $l^*\circ \pi^*$ factors through 
the augmentation $\e : A_{DR}(M) \to A_{DR}(\{m\})={\mathbb R}$ and then $j$ induces a morphism 
$k : \wedge V = A_{DR}(\{m\})\otimes_{A_{DR}(M)}R \to A_{DR}(P_m)$ of cochain algebras. 

\begin{prop}\label{prop:KSextension} Suppose that $N$ is simply connected. Then 
the morphism $k : \wedge V  \to A_{DR}(P_m)$ of cochain algebras is a quasi-isomorphism. 
\end{prop}

This result follows from \cite[20.3 Theorem]{H}. We prove Proposition \ref{prop:KSextension} by constructing the Leray--Serre spetral sequence 
and by applying the comparison theorem of spectral sequences. 

To this end, we first recall an isomorphism $a : A_{DR}(E_f) \to \Gamma(F)$ of cochain algebras in \cite[19.21 Lemma]{H} 
defined by $(a\psi)_\sigma = a_\sigma \psi$, where $a_\sigma$  denotes the composite 
\[
A_{DR}(E_f) \stackrel{\xi_\sigma}{\to} A_{DR}(P_\sigma) \to A_{DR}(\underline{P_\sigma}).
\]  
For the map $P_m \to \{m\}$, Lemma \ref{lem:DiffCoeff} enables us to obtain a local system over 
$K':=\{{S_\bullet^D(\{m\})}_{\text{sub}}\}$ of the form 
$F':= \{A_{DR}(\underline{(P_m)_\tau}) \}_{\tau \in K'}$. Observe that the inclusion $i : P_m \to E_f$ induces a morphism $i^* : F\to F'$ of local systems. 
Moreover, we have an isomorphism $a :  A_{DR}(P_m) \to \Gamma(F')$
by applying \cite[19.21 Lemma]{H}. 
Recall the quasi-isomorphism $i_F : \Gamma(F) \to \Gamma((A_{DR}^*)_\bullet\otimes F)$ 
which is defined by the inclusion $F_\sigma \to 1\otimes F_\sigma \subset 
A_{DR}({\mathbb A}^n)\otimes F_\sigma$ for $\sigma \in K_n$; see \cite[12.36, 13.5 and 13.12 Theorem]{H}. 
Here we regard $(A_{DR}^*)_\bullet$ as a local system over $K$ defined by 
$(A_{DR}^*)_{\bullet, \sigma} = (A_{DR}^*)({\mathbb A}^n)$ for $\sigma \in K_n$. 
Moreover, the map $\xi_F : \Gamma(A\otimes F) \to \Gamma(F)$ is defined by 
\[
(\xi_F(a\otimes \Phi))_\sigma = a_\sigma\cdot\Phi_\sigma, 
\]
where 
$\cdot$ denotes the multiplication on $(A_{DR}^*)_\bullet$. It is readily seen that $\xi_F$ is a left inverse of $i_F$ and so 
it is a surjective quasi-isomorphism. We observe that $\xi_F$ is a morphism of $A_{DR}(M)$-algebras. 
These maps gives a commutative diagram 
\[
\xymatrix@C38pt@R18pt{
A_{DR}(\{m\}) ={\mathbb R} \ar[d] \ar[dr]  \ar@/^0.9pc/[drr] \ar@/^1.4pc/[drrr] & & & \\
\wedge V  \ar[r]^-k  & A_{DR}(P_m) \ar[r]^-{a}_-{\cong} & \Gamma(F')  \ar@/_1.0pc/[r]_{i_{F'}}& 
\Gamma((A_{DR}^*)_\bullet\otimes F')\ar@{->>}@<0ex>[l]_-{\xi_{F'}}^-{\simeq}\\ 
A_{DR}(M)\otimes \wedge V \ar[r]_-{\simeq}^-j \ar[u]^q & A_{DR}(E_f) \ar[u]^-{l^*} \ar[r]_-{a}^-{\cong} & 
  \Gamma(F)  
  \ar[u]_{\Gamma(i^*)} & \Gamma((A_{DR}^*)_\bullet\otimes F). \ar@{->>}@<0ex>[l]^-{\xi_F}_-{\simeq} 
  \ar[u]_{\Gamma(1\otimes i^*)}\\
A_{DR}(M) \ar@/^3.5pc/[uuu] \ar[u] \ar[ur]^-{\nu^*} \ar@/_0.9pc/[urr] \ar@/_1.4pc/[urrr]& & & 
}
\]
containing the KS extension $j$, where the maps $i_{F'}$ and $\xi_{F'}$ are defined in the same way as $i_F$ and $\xi_F$, respectively. In particular, we have 
$\xi_{F'}\circ i_{F'} = id$. 
The Lifting lemma gives rise to a morphism $\Psi : A_{DR}(M)\otimes \wedge V \to \Gamma((A_{DR}^*)_\bullet\otimes F)$ of $A_{DR}(M)$-algebras 
with $\xi_F\circ \Psi = j \circ a$. More precisely, we define $\Psi$ by $\Psi(v) = i_F \circ j \circ a (v)$ for $v  \in V$. 
The commutativity of the three squares enables us to deduce that 
\[
a \circ k\circ q|_{1\otimes \wedge V} = \xi_{F'}\circ \Gamma(1\otimes i^*) \circ \Psi|_{1\otimes \wedge V}.
\eqnlabel{add-10}
\]

Define filtrations $G=\{G^p\}_{p\geq 0}$ of  $A_{DR}(M)\otimes \wedge V$ and $'G=\{'G^p\}_{p\geq 0}$ by 
$G^p = \sum_{i\geq p}A^i_{DR}(M) \otimes \wedge V$ and $'G^p = \Gamma (\sum_{i\geq p}(A^i_{DR})_\bullet \otimes F)$, respectively.  
Since the morphism $\Psi$ of cochain algebras over $A_{DR}(M)$ preserves the filtrations, it follows that the map 
induces a morphism 
$\{f_r\}_{r\geq 2} : \{E_r^{*,*}, d_r\} \to \{'E_r^{*,*}, 'd_r\}$ of spectral sequences constructed from the filtrations mentioned above; see \cite[(12.43)]{H}. 
Recall the integration map defined in (3.1); it follows from the proof of \cite[14.18 Theorem]{H} that integration induces a quasi-isomorphism 
\[
\int : \ \! \! 'E_1= \Gamma((A^*_{DR})_\bullet \otimes H(F)) \to  C^*(M; \mathcal{H}(L)),  
\]
where  $C^*(M; \mathcal{H}(L))$ denotes the cochain complex of $S^D_\bullet(M)_{\text{sub}}$ 
with the local coefficients induced by the local system $F$ and $H(F)$ is the local system of coefficients defined by 
$H(F)_\sigma := H(F_\sigma, d)$ for $\sigma \in K$.  
An important point of the proof is that $(A_{DR}^*)_\bullet$ and $(C_{PL})_\bullet$ are extendable; 
see \cite[12.37 Theorem]{H} and Section 3. 
  
\begin{rem} By the above argument, the Leray--Serre spectral sequence in Theorem \ref{thm:LSSS} gives an isomorphism 
$_{LS}E_2 \cong H^*(S^D_\bullet(M)_{\text{sub}})\otimes H^*(A_{DR}(L))$ as an algebra 
if $M$ is simply connected and the cohomology $H^*(S^D_\bullet(M)_{\text{sub}})$ is of finite type. 
\end{rem}

\begin{proof}[Proof of Proposition \ref{prop:KSextension}] Since $N$ is simply connected, it follows that the local system $F$ on ${S^D_\bullet(M)}_{\text{sub}}$ 
in Lemma \ref{lem:DiffCoeff} is simple. We observe that the action of $\pi_1(M)$ on $F$ is induced by that of $\pi_1(N)$. Then we see that $f_2$ is a morphism of free $H^*(M)$-modules.  Therefore if $f_2^{0,q}$ is an isomorphism for some $q$, then so is $f_2^{p,q}$ for any $p\geq 0$. 
It follows from the comparison theorem (\cite[17.17 Theorem]{H}) that $f_r^{p,q}$ is an isomorphism for any $r \geq 2$ and 
$p$, $q \geq 0$. The formula (5.5) implies that the isomorphism $f_2^{0, q}$ is nothing but the map 
\[
H(k) =  H(\int \circ i_{F'}\circ a\circ k) : H(\wedge V) \to H(F')=H(P_m). 
\]
This completes the proof. 
\end{proof}

We are ready to prove the main theorem of this section. 
\begin{proof}[Proof of Theorem \ref{thm:general_main}]
For a diffeological space $X$, we recall the quasi-isomorphism  $(j^*)^* : A^*_{DR}(X) \to A^*_{DR}(S^D_\bullet(X))=:A(X)$ in (5.2). 
Let $\Omega N \to PN\to N$ be the pullback of the evaluation map $(\e_0, \e_1) : N^I \to N\times N$ along the induction $s : N \to N\times N$ defined by 
$s(x)= (*, x)$, where $*$ denotes the base point of $N$.  
We have a commutative diagram of solid arrows 
\[
\xymatrix@C25pt@R3pt{
             & A^*_{DR}(\Omega N) \ar[rr]^{(j^*)^*}_-\simeq & & A(\Omega N) & \overline{B}(A) \ar[l]_{\alpha \circ  \mathsf{It}} 
             \ar@{..>}@/^8pt/[lld]_{\overline{\alpha}} \\
T \ar[ur]^-\kappa \ar@{.>}[rr]^(0.6){\overline{\beta}}& & T' \ar[ru]^-{\kappa'}& \\
             & A^*_{DR}(PN) \ar[uu]|{\hole} \ar[rr]|{\hole}^(0.4)\simeq & & A(PN) \ar[uu]_(0.3){A(i)} & \Omega^*(N)\otimes \overline{B}(A)
              \ar[l]_-{\alpha \circ  \mathsf{It}} \ar[uu] 
             \ar[uu] \ar@{..>}@/^8pt/[lld]_{\widetilde{\alpha}} \\
R \ar[uu] \ar[ur]^-\simeq \ar@{.>}[rr]^(0.6)\beta & & R'\ar[uu] \ar@{->>}[ru]_p^-\simeq & \\
             & A^*_{DR}(N) \ar[uu]|{\hole}_(0.4){\pi^*} \ar[rr]|{\hole}^(0.4)\simeq  & & A(N) \ar[uu]_(0.3){A(\pi)} & \Omega^*(N) \ar[l]_{\alpha} \ar[uu] \ar@/^8pt/[lld]_\alpha\\
A^*_{DR}(N)  \ar[uu]_(0.6)j \ar@{=}[ur] \ar[rr] & & A(N) \ar@{=}[ur] \ar[uu]_(0.7){j'} &  
}
\]
in which $j$ and $j'$ are KS extensions of $\pi^*$ and $A(\pi)$, respectively.  
Here $A$ denotes the cochain subalgebra of $\Omega^*(N)$ described in the paragraph before (5.1). 

We may assume that the quasi-isomorphism $p$ is a surjection by the surjective trick; see \cite[Section 12 (b)]{F-H-T}. 
By applying the Lifting lemma, we have a morphism $\beta : R \to R'$ which makes the two squares commutative.  
Then we have a morphism $\overline{\beta} : T:={\mathbb R}\otimes_{A_{DR}(N)}R \to T':={\mathbb R}\otimes_{A(N)}R'$ of cochain algebras. 
Moreover, the map $\beta$ is a quasi-isomorphism and hence so is $\overline{\beta}$ by \cite[Theorem 6.10]{F-H-T}. Proposition \ref{prop:KSextension} implies that $\kappa$ is a quasi-isomorphism and then so is $\kappa'$. Since the bar complex $\Omega^*(N)\otimes \overline{B}(A)$ is 
a semifree $\Omega(N)$-module, 
it follows from the Lifting lemma that there exist a morphism $\widetilde{\alpha} :  \Omega^*(N)\otimes \overline{B}(A) \to R'$ of $\Omega(N)$-modules and a morphism $\overline{\alpha} : \overline{B}(A) \to T'$ of differential graded modules which fit in the commutative diagram. 
Observe that $\overline{B}(A)\cong {\mathbb R}\otimes_{\Omega^*(N)}(\Omega^*(N)\otimes \overline{B}(A))$. 
The complex $\Omega^*(N)\otimes \overline{B}(A)$ is indeed a resolution of ${\mathbb R}$ and the diffeological space $PN$ is smoothly contractible. Then the map $\widetilde{\alpha}$ is a quasi-isomorphism. Since the factor map is a quasi-isomorphism by assumption, it follows from the comparison theorem \cite[7.1 Theorem]{H} that so is $\overline{\alpha}$. 
We see that $\alpha \circ  \mathsf{It} : \overline{B}(A) \to A(\Omega N)$ is a quasi-isomorphism. 

We apply the same argument to the pullback $\Omega N \to E_f \to M$ of the evaluation map $(\e_0, \e_1) : N^I \to N\times N$ 
along an induction $f: M \to N\times N$. 
Then in the diagram above, the bar complex $\Omega^*(N)\otimes \overline{B}(A)$ is also replaced with 
the complex $\Omega^*(M)\otimes_f \overline{B}(A)$.
The same argument as above with the comparison theorem enables us to conclude that $\alpha \circ  \mathsf{It} : \Omega^*(M)\otimes_f B(A) \to A(E_f)$ is 
a quasi-isomorphism. 
\end{proof} 

We conclude this section with a table which summarizes the simplicial objects used in this manuscript.  

The pair of a simplicial set in the first row and a simplicial cochain algebra in the first column gives a cochain algebra. 
These  cochain algebras are quasi-isomorphic to one another. In fact, the quasi-isomorphisms are induced by the inclusion 
$S^\infty_\bullet(X) \to S^D_\bullet(X)$, the restrictions $j^* : S^D(X)_\bullet \to  S^\infty_\bullet(X)$ and 
$j^* (A_{DR}^*)_\bullet \to \widetilde{(A_{DR}^*)}_\bullet$. More precisely, the results for the pairs in the second row  follow from Lemmas 
\ref{lem:acyclic}, \ref{lem:representable} and 
the commutativity of the same right square as in the diagram (4.2) in the proof of Proposition \ref{prop:alpha_beta}. 
We have the result for the pair in each column by considering the commutativity of the same right triangle as in (4.2).

\begin{table}[h]
{\small
  \begin{tabular}{|l|c|c|c|} \hline
                              & $S^D_\bullet(X)$ & ${S^D_\bullet(X)_{\text{sub}}}$ & $S^\infty_\bullet(X)$  \\ \hline
                                      &   This pair is used in                     &  We use this pair when      &    This is used in                 \\  
$(A_{DR}^*)_\bullet$     &  proving  the     &  constructing  the SSes       &  describing     \\   
                                       &   de Rham theorem   & in Theorems \ref{thm:LSSS} and \ref{thm:EMSS}    &  Proposition  \ref{prop:alpha_beta}  \\  \hline
\vspace*{-0.12cm}                           &  This case (1) is       &  This case (2) is      &  We use this pair  \\  
 $\widetilde{(A_{DR}^*)}_\bullet$  &   not used      &  not used    & when constructing $\alpha'$  \\  
                                       &  in our framework         &  in our framework     &  in Proposition  \ref{prop:alpha_beta} \\ \hline
     \end{tabular}
     }
     \caption{}
     \label{table1}
\end{table}

\vspace*{-0.5cm}        

In \cite{Kihara1}, Kihara has introduced {\it standard simplices} $\Delta^p_{\text{Ki}}$ for $p\geq 0$ in $\mathsf{Diff}$ whose underlying topological spaces are the standard ones in the category of topological spaces. With the simplices, it is proved that $\mathsf{Diff}$ admits a Quillen model category structure; see \cite[Theorem 1.3]{Kihara1}. For a diffeological space $X$, we can consider the complex associated with the singular simplex 
$S^D_p(X)_{\text{Ki}}$ consisting of smooth maps $\Delta^p_{\text{Ki}} \to X$, which is quasi-isomorphic to 
the complex for $S^D_\bullet(X)_{\text{sub}}$; see \cite[Remark 3.8]{Kihara}. Then the pairs 
(3)$:=(S^D_\bullet(X)_{\text{Ki}}, (A_{DR}^*)_\bullet)$ and 
(4)$:=(S^D_\bullet(X)_{\text{Ki}}, \widetilde{(A_{DR}^*)}_\bullet)$ give rise to the cochain algebras 
which are quasi-isomorphic to the cochain algebra for the pair (2). 
While it is possible to choose the pairs (1), (2), (3) and (4) when considering the cohomology algebras,  
we do not use them explicitly in this manuscript. 


\medskip
\noindent
{\it Acknowledgements.}
The author thanks Akinori Emoto and Norio Iwase for many valuable discussions on the simplicial cochain algebras 
concerning the de Rham theory for diffeological spaces.  
The author is also grateful to Hiroshi Kihara for his comments on Theorem \ref{thm:main}, which will lead the author to consider the problem when the factor map  
$\alpha$ in the theorem induces an isomorphism on the cohomology with a model structure on 
the category $\mathsf{Diff}$.  The author would also like to thank Dan Christensen for letting him realize the necessity of the smoothness of the CW-complexes in Theorem \ref{thm:main}. 

A part of this manuscript  has been written during his stay at Nesin Mathematics Village in which 
Diffeology, Categories and Toposes and Non-commutative Geometry Summer School has been held in the summer 2018. 
The author thanks Serap G\"urer, who was the organizer of the school, for her hospitality. 
Furthermore, the author greatly appreciates many constructive suggestions of the referee and John McCleary 
to improve this manuscript.
This work was partially supported by JSPS KAKENHI Grant Number 19H05495.

\section{Appendix}
\subsection{Appendix A: The acyclic model theorem for cochain complexes}\label{app0}

We recall the acyclic model theorem due to Bousfield and Gugenheim \cite{B-G}.  
\begin{defn}
Let $\C$ be a category and $\text{Ch}^*(\K)$ the category of cochain complexes over a field $\K$. A contravariant functor $F : \C \to \text{Ch}^*(\K)$ 
admits a {\it unit} if for each object $X$ in $\C$, there exists a morphism $\eta_X : \K \to F(X)$ in $\text{Ch}^*(\K)$.  
Let $\M$ be a set of objects in $\C$, which is called {\it models}. 
A functor $F$ with unit is {\it acyclic on models} $\M$ if for any $M$ in $\M$, there exists a morphism $\e_M : F(M) \to \K$ such that 
$\e_M \circ \eta_M \simeq id$ and $\eta_M\circ \e_M \simeq id$. 
\end{defn}

Let $F : \C \to \K\text{-Mod}$ be a functor from a category with models $\M$ 
to the category of vector spaces over $\K$. Then we define a contravariant functor $\widehat{F} :  \C \to \K\text{-Mod}$ by 
\[
\widehat{F}(X) :=\prod_{M \in \M, \sigma \in \C(M, X)}(F(M)\times \{\sigma \}),
\]
where for a morphism $f : X \to Y$ in $\C$, the morphism $\widehat{F}(f) :  \widehat{F}(Y) \to \widehat{F}(X)$ is defined by 
$\widehat{F}(f)\{m_\sigma, \sigma\} =\{m_{f\tau}, \tau\}$. 
Moreover, we define a natural transformation $\Phi : F \to \widehat{F}$ by $\Phi_X(u) =\{F(x)u, x\}$. 
We say that $F$ is {\it corepresentable} on the models $\M$ if there exists a natural transformation 
$\Psi : \widehat{F} \to F$ such that $\Psi\circ \Phi = id_F$. 

\begin{thm}\label{thm:AcyclicModels} \cite[2.4 Proposition]{B-G}
Let $\C$ be a category with models $\M$. 
Let $K_1$ and $K_2$ be contravariant functors from $\C$ to $\text{\em Ch}^*(\K)$ with units 
$\eta : \K \to K_1^0, K_2^0$. Here $\K$ denotes the constant functor defined by $\K(X)= \K$.  
Suppose that $K_1$ is acyclic on models $\M$ and $U_k\circ K_2$ is corepresentable on the models for any $k$, where $U_k$ denotes the forgetful functor to 
$\K\text{\em -Mod}$ on the degree $k$. Then there exists a natural transformation $T : K_1 \to K_2$ which preserves the unit. 
Moreover any two such natural transformations are naturally homotopic. 
\end{thm}

We here consider III) in Section \ref{sub4.1} more precisely.  
In Theorem \ref{thm:AcyclicModels}, we take the category $\mathsf{Diff}$ as $\C$ and then put $K_1=\Omega^*(\text{-})$ and 
$K_2= C^*(S_\bullet^D( \text{-} ))$. Let $\M$ be the subset of objects in $\C$ consisting of the affine spaces ${\mathbb A}^n$ for any $n\geq0$. Then the category 
$\mathsf{Diff}$ is regarded as a category with models $\M$. The Poincar\'e lemma for diffeological spaces implies that the functor $\Omega^*(\text{-})$ is acyclic for $\M$; 
see \cite[6.83]{IZ}. 

For a non-negative integer $k\geq 0$, 
we define a map 
\[
\Psi_X : \widehat{C^k(S_\bullet^D(X))} 
:= \prod_{{\mathbb A}^n \in \M, \sigma \in C^{\infty}({\mathbb A}^n, X))} (C^k(S_\bullet^D({\mathbb A}^n)\times \{\sigma\})\to C^k(S_\bullet^D(X))
\]
by $\Psi_X(\{m_\sigma, \sigma\})(\tau) = m_\tau(id_{{\mathbb A}^k})$, where $\tau \in S_k^D(X)$. Then $\Psi_{\text{--}}$ is a natural transformation. In fact, we see that 
for a smooth map $f : X \to Y$ and $u \in S_k^D(X)$, 
\[
\Psi_X(\widehat{C^k(S_k^D(f))} \{m_\sigma, \sigma\})(u) = \Psi_X\{m_{f\tau}, \tau\}(u) = m_{fu}(id_{{\mathbb A}^k}) \ \ \text{and}
\]
\[
((C^kS_\bullet^D)(f))(\Psi_Y\{m_\sigma, \sigma\})(u) = \Psi_Y\{m_\sigma, \sigma\}(fu) = m_{fu}(id_{{\mathbb A}^k}). 
\]
Since $\Phi_X(u)= \{C^k(S_\bullet^D(\sigma))u, \sigma\} $
for $u \in C^k(S_\bullet^D(X))$ by definition, 
it follows that 
\begin{align*}
(\Psi_X\Phi_X(u))(\tau) &= \Psi_X(\{C^k(S_\bullet^D(\sigma))u, \sigma\})(\tau) = C^k(S_\bullet^D(\tau))u(id_{{\mathbb A}^k}) \\
&= u(\tau \circ id_{{\mathbb A}^k})=u(\tau)
\end{align*}
for $\tau \in S_k^D(X)$. Then we have $\Psi\Phi = id$ and hence $C^k(S_\bullet^D( \text{-} ))$ is corepresentable. 
Theorem \ref{thm:AcyclicModels} enables us to deduce 
the homotopy commutativity of the right square in Theorem \ref{thm:main}.

\subsection{Appendix B}\label{AppB}
We begin with the definition of a smooth CW complex in the sense of Iwase and Izumida.

\begin{defn}\label{defn:CW}\cite[Appendix]{I-I}
A {\it smooth CW complex} $\{K, \{K^{(n)}\}_{n\geq -1}\}$ is a diffeological space built up from $K^{-1}=\varnothing$ by inductively attaching $n$-balls $\{B_j^n\}_{j\in J_n}$ by smooth maps from their boundary spheres $\{S_j^{n-1}\}_{j\in J_n}$ to the $n-1$ skeleton $K^{(n-1)}$ to obtain 
the $n$ skeleton $K^{(n)}$. Here the diffeology of balls and boundary spheres is given by their manifold structures, and $K=\text{colim}\  K^{(n)}$. 
A smooth CW complex $\{K, \{K^{(n)}\}_{n\geq -1}\}$ is {\it finite dimensional} if the set of skeletons is finite.
\end{defn}

To define a stratifold,  
we recall a differential space in the sense of Sikorski \cite{Sik}. 
\begin{defn} \label{defn:differential_space}
A {\it  differential space} is a pair $(S, \C)$ consisting of a topological space $S$ and an $\mathbb{R}$-subalgebra $\C$
of the $\mathbb{R}$-algebra $C^0(S)$ of continuous real-valued functions on $S$, which is assumed to be {\it locally detectable} and 
$C^\infty$-{\it closed}.  

\medskip
Local detectability means that $f \in \C$ if and only if 
for any $x \in S$, there exist an open neighborhood $U$ of $x$ and an element $g \in \C$ such that 
$f|_U = g|_U$.  

\medskip
$C^\infty$-closedness means that for each $n\geq 1$, each $n$-tuple $(f_1, ..., f_n)$ of maps in $\C$ and each smooth map 
$g : \mathbb{R}^n \to \mathbb{R}$, the composite  
$h : S \to \mathbb{R}$ defined by $h(x) = g(f_1(x), ...., f_n(x))$ belongs to $\C$.
\end{defn}

Let $(S, \C)$ be a differential space and $x \in S$. The vector space consisting of derivations on the $\mathbb{R}$-algebra 
$\C_x$ of the germs at $x$ is denoted by $T_xS$, which is called the {\it tangent space} 
of the differential space  at $x$; see \cite[Chapter 1, section 3]{Kreck}. 

\begin{defn} \label{defn:stratifold} 
A {\it stratifold} is a differential space $(S, \C)$ such that the following four conditions hold: 

\begin{enumerate}
\item
$S$ is a locally compact Hausdorff space with countable basis; 
\item
the {\it skeleta} $sk_k(S):= \{x \in S \mid \text{dim } \!T_xS\leq k\}$ are closed in $S$;
\item
for each $x \in S$ and open neighborhood $U$ of $x$ in $S$, 
there exists a {\it bump function} at $x$ subordinate to $U$; that is, a non-negative function $\rho \in \C$ such that 
$\rho(x)\neq 0$ and such that the support $\text{supp }\!\rho :=\overline{\{p \in S \mid \rho(p) \neq 0\}}$ is contained in $U$; 
\item
the {\it strata} $S^k := sk_k(S) - sk_{k-1}(S)$ are $k$-dimensional smooth manifolds such that restriction along 
$i : S^k \hookrightarrow S$ induces an isomorphism of stalks 
$
i^* : \C_x \stackrel{\cong}{\to} C^\infty(S^k)_x
$
for each $x \in S^k$.
\end{enumerate}
\end{defn}

By definition, a continuous map $f : (S, \C) \to (S', \C')$ is a morphism of stratifolds if $\phi \circ f \in \C$ 
for any $\phi \in \C'$. We denote by $\mathsf{Stfd}$ the category of stratifolds. 

A parametrized stratifold is constructed from a manifold attaching other finite 
manifolds with boundaries. More precisely, let $(S, \C)$ be 
a stratifold of dimension $n$ and 
$W$ a $k$-dimensional manifold with compact boundary $\partial W$ endowed with a collar 
$c : \partial W \times [0,  \varepsilon) \to W$. Suppose that $k > n$. Let $f : \partial W \to S$ be a morphism of stratifolds. 
Let 
$S'$ be the identification space $S'=S\cup_f W$ and we define a subalgebra $C'$ of  $C^0(S')$ by 
\[
C'=\Set{ g : S' \to \R  | g_{| S} \in \C, \text{$gc(w, t)=gf(w)$ for $w \in \partial W$} }. 
\]
Then, the pair $(S',  \C')$ is a stratifold; see \cite[Example 9]{Kreck}. 
We call a stratifold a {\it parametrized} stratifold if it is built up applying the construction above inductively to a finite sequence 
$W_0$, $(W_1, f_1)$, ..., $(W_n, f_n)$  of $i$-dimensional manifolds $W_i$ with compact boundaries equipped with collars and morphisms $f_i$ form $\partial W_i$ to the stratifold constructed inductively by 
$W_0$, $(W_1, f_1)$, ..., $(W_{i-1}, f_{i-1})$.

Let $\mathsf{Diff}$ be the category of diffeological spaces. There is a functor  
$k :  \mathsf{Stfd} \to \mathsf{Diff}$
defined by 
$k(S, \C) = (S, \D_\C)$ and $k(f) = f$ for a morphism $f : S \to S'$ of stratifolds, where 
\[
\D_\C:=\Set{u : U \to S | 
\begin{array}{l}
\text{$U :$ open in ${\R^q}, q \geq 0$,} \\
\text{$\phi\circ u \in C^\infty(U)$ for any  $\phi \in \C$} 
\end{array} }. 
\]
We observe that a plot in $\D_\C$ is a set map and that the functor $k$ is faithful; see \cite[Section 5]{A-K} for the details. 

\begin{rem}\label{rem:stratifolds}
The category $\mathsf{Diff}$ has the large class of objects, namely diffeological spaces. Then in order to consider whether properties of manifolds are inherited 
to  more general objects in $\mathsf{Diff}$, it is important to restrict the category to a particular subcategory containing $\mathsf{Mfd}$.
The wide subcategory of $\mathsf{Diff}$ defined by the image of the functor $k$ and {\it admissible maps} is a candidate for such a subcategory; 
see \cite[Section 5]{A-K}. Indeed, a stratifold admits a partition of unity and the Serre-Swan theorem holds in $\mathsf{Stfd}$ and then in the wide subcategory. 
We expect to consider tangent spaces and tangent bundles in the subcategory by applying the results in \cite{Hector, C-W_tangentSp}.  
\end{rem}

\begin{lem}\label{lem:D-top} Let $(S, \C)$ be a stratifold. 
An open set of the underlying topological space $S$ is a $D$-open set of the diffeological space $k(S, \C)$.
\end{lem}

\begin{proof}
Let $u$ be an element in $\D_\C$ with domain $U$. Then $u : U \to k(S, \C)$ is a smooth map in the sense of diffeology. 
In fact, for any plot $p : V \to U$ 
of the diffeology $U$ and  for any $\phi \in \C$, we see that $\phi \circ u_*(p) = (\phi \circ u) \circ p$ is in $C^\infty(V)$ 
and hence $u_*(p)$ is in $\D_\C$. Since $U$ is a manifold, it follows from \cite[Proposition 5.1]{A-K} that $u$ is a morphism in 
$\mathsf{Stfd}$. In particular, the plot $u$ is continuous. It turns out that, by definition, each open set of $S$ is $D$-open. 
\end{proof}

We here summarize categories and functors related to our work.  
\[
\xymatrix@C50pt@R12pt{
& & \mathsf{Sets}^{\Delta^{op}}   \ar@<1ex>[d]^-{| \ |_D}&  \\
\mathsf{Mfd} \ar[r]_{\text{\tiny fully faithful}}^j \ar@/^2.0pc/[rr]^{\ell : \text{fully faithful}} \ar[rd]_{{\mathcal Y}}& \mathsf{Stfd} \ar[r]^-k &\mathsf{Diff} 
\ar@<1ex>[r]^-{D}_-{\bot} \ar@<1.5ex>[u]^-{S^D}_{\vdash}   \ar@<1.6ex>[dl]^-{\int} 
& \mathsf{Top}, \ar@<1.2ex>[l]^-{C}   \\
&\mathsf{Stacks}_{\mathsf{Mfd}}  \ar@<0.5ex>[ur]^(0.4){Coarse}_(0.5){\dashv}& C\text{-}\mathsf{Diff} \ar@<1ex>[r]^-{D}_-{\simeq} \ar[u] &  \Delta\text{-}\mathsf{Top}\ar@<1ex>[l]^-{C}\ar[u]
}
\eqnlabel{add-10}
\]
The category $\mathsf{Mfd}$ also embeds into the category of differentiable stacks $\mathsf{Stacks}_{\mathsf{Mfd}}$ with the Yoneda embedding ${\mathcal Y}$. Moreover, the category $\mathsf{Stacks}_{\mathsf{Mfd}}$ connects with 
$\mathsf{Diff}$ using the Grothendieck construction functor $\int$ and the coarse moduli space functor {\it Coarse}, which are adjoints to each other; see \cite{W-W} for more details.  

The $D$-topology of diffeological spaces gives rise to the functor $D: \mathsf{Diff} \to  \mathsf{Top}$. For a topological space $X$, all continuous maps make 
the diffeology $\D^X$ of the underlying set $X$. Thus we have the functor $C$ mentioned in the diagram above; 
see Remark \ref{rem:Delta-generated_Top} for an important property of the adjoint pair. 
The realization of a simplicial set in $\mathsf{Diff}$ with 
affine spaces ${\mathbb A}^n$ for $n\geq 0$ gives the realization functor $|\ |_D$. 
The results \cite[Propositions 4.14 and 4.15]{C-W} assert that the functors 
$S^D \circ C$ and $D \circ | \ |_D$ coincide with the usual singular simplex functor and the realization functor up to weak equivalence, respectively.  
We refer the reader to \cite{S-Y-H} and \cite{C-W} for more properties of the adjoint pairs $(D, C)$ and $(| \ |_D, S^D)$. 

\begin{rem}\label{rem:Delta-generated_Top}
The functors $C$ and $D$ give rise to an equivalence between appropriate full subcategories of $\mathsf{Diff}$ and 
$\mathsf{Top}$. To describe this more precisely, we recall that the unit and counit induce isomorphisms 
$\eta_{CX} : CX \stackrel{\cong}{\to} CDCX$ and 
$\e_{DY} : DCDY \stackrel{\cong}{\to} DY$; see \cite[Propositios 3.3]{C-S-W} and also \cite{S-Y-H}. 
Let $C\text{-}\mathsf{Diff}$ be 
the full subcategory of $\mathsf{Diff}$ consisting of objects isomorphic to diffeological spaces in the image of $C$ and 
$\Delta\text{-}\mathsf{Top}$ the full subcategory of $\mathsf{Diff}$ consisting 
of objects isomorphic to topological spaces in the image of $D$. We observe that the objects in the image of $D$ are exactly the $\Delta$-generated topological spaces; see \cite[Proposition 3.10]{C-S-W}. 
The result \cite[Lemma II. 6.4]{M-M} implies that the functors are restricted to equivalences between 
$C\text{-}\mathsf{Diff}$ and $\Delta\text{-}\mathsf{Top}$. It is worth mentioning that all CW-complexes are included in $\Delta\text{-}\mathsf{Top}$; 
see \cite[Corollary 3.4]{S-Y-H}. 
\end{rem}

\begin{rem}\label{rem:CDC}
Let $X$ be in the category $C\text{-}\mathsf{Diff}$. Since the unit $\eta_X : X \to CDX$ is an isomorphism, it follows that 
the map $\mathsf{Diff}(\Delta_\text{sub}^n, X) \to \mathsf{Diff}(\Delta_\text{sub}^n, CDX)$ induced by the unit is bijective 
and hence so is the composite $\mathsf{Diff}(\Delta_\text{sub}^n, X) \to \mathsf{Top}(\Delta^n, DX)$ of the maps in (4.1). 
Thus 
the composite gives rise to an isomorphism 
$H(C^*({S^D_\bullet(X)}_{\text{sub}})) \cong H^*(DX, {\mathbb R})$ 
for each object $X$ in $C\text{-}\mathsf{Diff}$, where $H^*(\text{-}, {\mathbb R})$ denotes the singular cohomology with coefficients in ${\mathbb R}$.  In particular, we have an isomorphism 
$H(C^*({S^D_\bullet(CZ)}_{\text{sub}})) \cong H^*(Z, {\mathbb R})$  for a CW-complex $Z$.
\end{rem}

\begin{rem}
Let $(S, \C)$ be a stratifold. Then it follows from \cite[Corollary 5.2]{A-K} that 
$S^D_\bullet(k(S, \C)) \cong \mathsf{Stfd}({\mathbb A}^\bullet, (S, \C))$ as a simplicial set. 
Corollary \ref{cor:main} implies that 
that the de Rham cohomology of $k(S, \C)$ is isomorphic to the cohomology of the chain complex induced by 
$\mathsf{Stfd}({\mathbb A}^\bullet, (S, \C))$ as an algebra. 
\end{rem}

\subsection{Appendix C}\label{AppC} 
Let $(X, \D^X)$ be a diffeological space. 
We recall the factor map $\alpha : \Omega^*(X) \to A^*(X):=A_{DR}(S^D_\bullet (X))$ defined in Section \ref{sect2}. 
In this section, we prove the following proposition. 

\begin{prop}\label{injectivity_of_H(a)} For each diffeological space $X$, the map $H^1(\alpha) : H^1(\Omega^*(X)) \to H^1(A^*(X))$ induced by 
the factor map $\alpha$ is injective. 
\end{prop}

\noindent
Moreover, we shall introduce an obstruction for the map $H^*(\alpha)$ to be injective. 

\begin{proof}[Proof of Proposition \ref{injectivity_of_H(a)}] 
We consider the map $C_1(S^D_\bullet(X)) \to C_{\text{cube}, 1}(X)$ induced by the inverse of the projection $pr_1 : {\mathbb A}^1 \to {\mathbb R}^1$ on the first factor, which is a diffeomorphism. Then the map gives a chain map $\widetilde{m} : C_*(S^D_\bullet(X)) \to C_{\text{cube}, *}(X)$ for $* \leq 1$. Moreover, the method of acyclic models extends $\widetilde{m}$ to a chain map $\widetilde{m} =\{ \widetilde{m}_i\}_{i\geq 0}$ defined on all of the degrees. We write $m$ for the dual to the chain map $\widetilde{m}$. Then by the same argument as in Section \ref{sub4.1} III)  with Theorem  \ref{thm:AcyclicModels}, we have a homotopy commutative diagram 
  \[
 \xymatrix@C25pt@R15pt{
 A^*(X) \ar[d]_{\int'}& \Omega^*(X) \ar[l]_-{\alpha} \ar[d]^{\int^{IZ}} \\
 C^*(S^D_\bullet(X)) & \ar[l]_-{m} C_{\text{cube}}^*(X),
 }
 \eqnlabel{add-11}
 \]
 where $\int'$ denotes the composite of the integration map $\int$ defined in Section \ref{sect3.2} and the isomorphism $\nu : C^*_{PL}(S^D_\bullet(X)) \stackrel{\cong}{\to} 
 C^*(S^D_\bullet(X))$ mentioned in Section \ref{sect2}. 
 Moreover, the definition of $\widetilde{m}_1$ enables us to obtain a commutative diagram 
 \[
 \xymatrix@C20pt@R12pt{
 H_1(S^D_\bullet(X)) \ar[rr]^{H(\widetilde{m}_1)} & & H_1(C_{\text{cube}, *}(X)) \\
 & \pi_1(X) \ar[ur]_{\Theta} \ar[ul]^{\Theta'}, 
 }
 \]
 where $\Theta$ and $\Theta'$ denote the Hurewicz maps. Composing 
 the diagram (6.2) on the first cohomology with the triangle obtained by dualizing the diagram above, we have a commutative diagram 
 \[
 \xymatrix@C20pt@R12pt{
 H^1(A^*(X))  \ar[rd]_-{\int''} & & H^1(\Omega^*(X)) \ar[ll]_{H^1(\alpha)} \ar[ld]^-{\widetilde{\int^{IZ}}}\\
 & \text{Hom}(\pi_1(X), {\mathbb R})  
 }
 \]
in which $\widetilde{\int^{IZ}}$ is nothing but the first de Rham homomorphism defined in \cite[6. 74]{IZ}. It follows from \cite[6.89]{IZ} that the de Rham homomorphism is a 
monomorphism in general. The commutativity yields the result. 
\end{proof}

The \v{C}ech--de Rham  spectral sequence defined 
in \cite{IZ_Cech, IZ_2019} for a diffeological space $X$ contains the de Rham cohomology $H^*(\Omega^*(X))$ of Souriau  in the edge of the $E_2$-term. 
After briefly describing the spectral sequence, we shall deduce that the injectivity of the edge homomorphism gives an obstruction for the map $H^*(\alpha)$ to be injective. 

Let $(X, \D^X)$ be a diffeological space and ${\mathcal G}$ a generating family of $\D^X$ in the sense of \cite[1.65]{IZ}. 
We may assume that the domain of each plot in ${\mathcal G}$ is a ball in ${\mathbb R}^N$ for some $N$. 
Then we define the {\it nebula} ${\mathcal N}_X$ of $X$ associated with ${\mathcal G}$ by 
\[
{\mathcal N}_X := \coprod_{\varphi \in {\mathcal G}} \big(\{\varphi\} \times \text{dom}(\varphi)\big)
\]
with sum diffeology, where $\text{dom}(\varphi)$ denotes the domain of the plot $\varphi$. 
It is readily seen that the evaluation map $ev : {\mathcal N}_X \to X$ defined by $ev(\varphi, r)=\varphi(r)$ is smooth. 
The {\it gauge monoid} $\mathsf{M}$ is a submonoid of the monoid of endomorphisms on the nebula ${\mathcal N}_X$ defined by 
\[
\mathsf{M} :=\{f \in C^\infty({\mathcal N}_X, {\mathcal N}_X) \mid ev\circ f = ev \ \text{and} \ 
\sharp \ \! \text{Supp} \ \!  f < \infty \},
\] 
where $\text{Supp} \  \!f := \{ \varphi \in {\mathcal G} \mid 
f|_{\{\varphi \} \times \text{dom}(\varphi)} \neq 1_{\{\varphi \} \times \text{dom}(\varphi)} \}$. 
Then the original de Rham complex $\Omega^*({\mathcal N}_X)$ is a left $\K\mathsf{M}^{\text{op}}$-module whose actions are defined by $f^*$ induced by an endomorphism $f \in {\mathcal N}_X$.  
Moreover, the complex $\Omega^*({\mathcal N}_X)$ is regarded as a two sided $\K\mathsf{M}^{\text{op}}$-module for which the right module structure is trivial. 
Then we have the Hochschild complex $C^{*,*}=\{C^{p,q}, \delta, d_\Omega\}_{p,q \geq 0}$ with 
\[
C^{p,q} = 
\text{Hom}_{\K\mathsf{M}^{\text{op}}\otimes\K\mathsf{M}}
(\K\mathsf{M}^{\text{op}}\otimes (\K\mathsf{M}^{\text{op}})^{\otimes p}\otimes\K\mathsf{M}, \Omega^q({\mathcal N}_X))\cong 
\text{map}(\mathsf{M}^p, \Omega^q({\mathcal N}_X)),
\]
where the horizontal map $\delta$ is the Hochshcild differential and the vertical map $d_\Omega$ is 
induced by the de Rham differential on $\Omega^*({\mathcal N}_X)$; see \cite[Subsection 12]{IZ_2019}. 
The total complex $\text{Tot} \ C^{*,*}$ has the horizontal filtration $F^* = \{F^j\}_{j\geq 0}$ defined by $F^j = \oplus_{q\geq j}C^{*,q}$. Then the filtration gives rise to a first quadrant spectral sequence $\{{}^d\!E_r^{*, *}, d_r\}$ converging to $H^*(\text{Tot} \ C^{*,*})$ with 
\[
E_2^{p,q} \cong H^q( HH^p(\K\mathsf{M}^{\text{op}}, \Omega^*({\mathcal N}_X)), d_\Omega), 
\]
which is called the \v{C}ech--de Rham  spectral sequence. 

The result \cite[Proposition in Subsection 9]{IZ_2019} implies that for a diffeological space $X$, 
the evaluation map induces an isomorphism $ev^* : \Omega^*(X) \stackrel{\cong}{\to} \Omega^*({\mathcal N}_X)^\mathsf{M}$ of cochain algebras, 
where $\Omega^*({\mathcal N}_X)^\mathsf{M}$ denotes the invariant subcomplex of $\Omega^*({\mathcal N}_X)$. 
Since the kernel of the map $\delta : C^{0, q} \to C^{1, q}$ is nothing but the complex $\Omega^*({\mathcal N}_X)^\mathsf{M}$, it follows that 
the edge homomorphism $$edge_1:= H(ev^*) : H^*(\Omega^*(X)) \stackrel{\cong}{\longrightarrow} {}^d\!E_2^{0, *}$$ is an isomorphism. 

We consider the vertical filtration of the total complex $\text{Tot} \ C^{*,*}$  and the spectral sequence $\{{}^\delta\!E_r^{*,*}, d_r\}$ 
associated with the filtration. Then the Poincar\'e lemma for the original de Rham complex yields that ${}^\delta\!E_r^{p,q}=0$ for $q >0$ and hence 
the target of $\{{}^\delta\!E_r^{*,*}, d_r\}$ is the Hochschild cohomology (the \v{C}ech cohomology)
$\check{H}(X)=HH^*(\K\mathsf{M}, \text{map}({\mathcal G}, \K))$; see \cite[Sections III and IV]{IZ_2019}. Moreover, we see that 
the edge homomorphism 
\[
edge_2 := H(inc_*) : \check{H}(X)  \stackrel{\cong}{\longrightarrow} H^*(\text{Tot} \ C^{*,*})
\]
is an isomorphism, where 
$inc_*$ is the map induced by the inclusion $ \text{map}({\mathcal G}, \K) \to \Omega^0({\mathcal N}_X)$ to the constant functions on ${\mathcal N}_X$.  

In the constructions of the two spectral sequence above, it is possible to replace the de Rham complex $\Omega^*(X)$ 
with the singular de Rham complex $A^*(X)$. 
While the map $edge_1 : H^*(A^*(X)) \to {}^{d'}\!E_2^{0, *}$ for $A^*(X)$ is merely a morphism of algebras, the map $edge_2$ for $A^*(X)$ is an isomorphism. 
In fact, Lemma \ref{lem:acyclic} and Theorem \ref{thm:main} imply that the Poincar\'e lemma for the singular de Rham complex holds. 
Since the factor map $\alpha$ gives rise to a natural transformation $\Omega^*( \ ) \to A^*(\ )$ 
(see Proposition \ref{prop:alpha_beta} (ii)),  it follows that the map induces a morphism 
$\{f(\alpha)_r\} : \{ {}^d\!E_r^{*, *}, d_r\} \to \{{}^{d'}\!E_r^{*, *}, d_r\}$ 
of spectral sequences.  
As a consequence, we have a commutative diagram 
\[
 \xymatrix@C25pt@R5pt{
 H^*(\Omega^*(X)) \ar[r]^-{edge_1}_-{\cong} \ar[dd]_{H(\alpha)} & 
  {}^d\!E_2^{0, *} \ar@{->>}[r] \ar[dd]_{f(\alpha)_2}
 &  {}^d\!E_\infty^{0, *} \ar@{>->}[r]  \ar[dd]_{f(\alpha)_\infty} & H^*(\text{Tot} \ C^{*,*})  \ar[dd]_{H^*(\text{Tot}(\alpha))} &\\
 & & & &  \check{H}^*(X), \ar[lu]_-{edge_2}^-{\cong}  \ar[ld]^-{edge_2}_-{\cong} \\
 H^*(A^*(X)) \ar[r]^-{edge_1} &  {}^{d'}\!E_2^{0, *} \ar@{->>}[r] 
 &  {}^{d'}\!E_\infty^{0, *} \ar@{>->}[r] & H^*(\text{Tot} \ 'C^{*,*}) & 
 }
  \eqnlabel{add-12}
\] 
where $\text{Tot}(\alpha)$ denotes the cochain map induced by the morphism $C^{*,*} \to \! \ 'C^{*,*}$ of double complexes which $\alpha : A^*(X) \to \Omega^*(X)$ gives. In particular, it follows that $H^q(\text{Tot}(\alpha))$ is an isomorphism for any $q$. We call the composite of maps in the first line in the diagram (6.3) the 
{\it edge homomorphism} of the \v{C}ech--de Rham spectral sequence for $X$. 
Then the commutative diagram above enables us to deduce the following proposition. 

\begin{prop}\label{prop:obstruction_for_Inj}
If the edge homomorphism 
$H^q(\Omega^*(X)) \to \check{H}^q(X)$ is invective, then the map $H^q(\alpha)$ is injective.  
\end{prop}

In the first quadrant spectral sequence $\{{}^d\!E_r^{*,*}, d_r\}$, the differential $d_r$ is of degree $(-r+1, r)$, 
namely $d_r:  {}^d\!E_r^{p,q} \to {}^d\!E_r^{p-r+1, q+r}$. 
Observe the grading of the filtration defined in \cite{IZ_2019}. It is readily seen that the sufficient condition in 
Proposition \ref{prop:obstruction_for_Inj} is equivalent to saying that 
every elements in ${}^d\!E_2^{0, q}$ in the  \v{C}ech--de Rham  spectral sequence is non-exact. 
By degree reasons, we see that each element in ${}^d\!E_2^{0, 1}$ is non-exact. Then Proposition \ref{prop:obstruction_for_Inj} gives 
another proof of Proposition \ref{injectivity_of_H(a)}. 

\begin{rem}
Iglesias-Zemmour has introduced a linear map $b_q : H^q(\Omega^*(X)) \to \check{H}^q(X)$, 
which is called the $q$th {\it \v{C}ech--de Rham homomorphism}, in \cite[Subsection 12]{IZ_2019}. By definition, we see that 
the map $b_k$ is nothing but the edge homomorphism for $X$ mentioned above.
\end{rem}

\end{document}